\documentclass[11pt]{amsart}

\usepackage{amscd}
\usepackage{amsmath, amssymb, comment}
\usepackage{amsfonts}
\usepackage{enumerate}
\usepackage{color}
\usepackage{url}
\newcommand{\de}{\partial}

\newcommand{\ddbar}{\sqrt{-1} \partial \overline{\partial}}
\newcommand{\Ric}{\mathrm{Ric}}
\newcommand{\ov}[1]{\overline{#1}}

\newcommand{\tr}[2]{\textrm{tr}_{#1}{#2}}
\newcommand{\ti}[1]{\widetilde{#1}}
\newcommand{\vp}{\varphi}

\newcommand{\ve}{\varepsilon}
\newcommand{\e}{\varepsilon}

\newcommand{\p}{\partial}
\newcommand{\pbar}{\overline{\partial}}

\renewcommand{\leq}{\leqslant}
\renewcommand{\geq}{\geqslant}

\newcommand{\be}{\begin{equation}}
\newcommand{\ee}{\end{equation}}

\newcommand{\PSH}{\textrm{PSH}}

\newcommand{\R}{\mathbb{R}}
\newcommand{\C}{\mathbb{C}}

\begin{document}
\newcounter{remark}
\newcounter{theor}
\setcounter{theor}{1}
\newtheorem{claim}{Claim}
\newtheorem{theorem}{Theorem}[section]
\newtheorem{lemma}[theorem]{Lemma}
\newtheorem{corollary}[theorem]{Corollary}
\newtheorem{proposition}[theorem]{Proposition}
\newtheorem{question}{question}[section]
\newtheorem{defn}{Definition}[theor]
\newtheorem{remark}[theorem]{Remark}

\numberwithin{equation}{section}

\title{$C^{1,1}$ regularity of geodesics of singular K\"{a}hler metrics}

\begin{abstract}We show the optimal $C^{1,1}$ regularity of geodesics in nef and big cohomology class on K\"ahler manifolds away from the non-K\"ahler locus, assuming sufficiently regular initial data. As a special case, we prove the $C^{1,1}$ regularity of geodesics of K\"ahler metrics on compact K\"ahler varieties away from the singular locus. Our main novelty is an improved boundary estimate for the complex Monge-Amp\`ere equation that does not require strict positivity of the reference form near the boundary. We also discuss the case of some special geodesic rays.
\end{abstract}

\author[J. Chu]{Jianchun Chu}
\address{Department of Mathematics, Northwestern University, 2033 Sheridan Road, Evanston, IL 60208}
\email{jianchun@math.northwestern.edu}

\author[N. McCleerey]{Nicholas McCleerey}
\address{Department of Mathematics, Northwestern University, 2033 Sheridan Road, Evanston, IL 60208}
\email{njm2@math.northwestern.edu}
\thanks{Partially supported by NSF RTG grant DMS-1502632.}


\maketitle

\section{Introduction}

Recall that if $(X^n,\omega_X)$ is a compact K\"ahler manifold without boundary, then the Mabuchi geodesic between two K\"ahler metrics cohomologous to $\omega_X$ \cite{Mab} can be given as the solution to the following Dirichlet problem:
\[\begin{cases}
(\pi^*\omega_X + i\p\pbar\Phi)^{n+1} = 0, \\
\Phi|_{\p(X\times A)} = \vp \in C^{\infty}(\overline{X\times A})\cap\PSH(X\times A,\pi^*\omega_X),
\end{cases}\]
where $A\subset \C$ is an annulus, $\pi: X\times A\rightarrow X$ is the projection map, and $\vp$ is a smooth, rotationally symmetric function determined solely by the two given K\"ahler metrics \cite{Se, Don}. There has been a great deal of work done to establish regularity and positivity properties for such geodesics -- see \cite{Ber2,Blo,CTW1,CTW,CZ,Da1,He,PS,PS2,RWN,To}. In particular, the recent work of Chu-Tosatti-Weinkove \cite{CTW1} has established that solutions are $C^{1,1}$ regular (based on their earlier work \cite{CTW0}), which is known to be optimal by examples of Lempert-Vivas \cite{LV}, Darvas-Lempert \cite{DL}, and Darvas \cite{Da}.

After adding and subtracting a K\"ahler potential pulled back from the annulus (i.e. a solution to the Laplace equation with right-hand side $1$), the geodesic equation can be rewritten to have a reference form which is an actual K\"ahler metric $\ti\omega$ on $X\times A$, without changing the boundary data:
\[\begin{cases}
(\ti \omega + i\p\pbar\ti\Phi)^{n+1} = 0, \\
\ti \Phi|_{\p(X\times A)} = \vp.
\end{cases}\]

Thus, if one is interested in studying geodesics of K\"ahler metrics, one is naturally let to consider the classical Dirichlet problem for the homogeneous Monge-Amp\`ere operator:
\begin{equation}\label{Chpt2_1}
\begin{cases}
(\omega + i\p\pbar V)^{n} = 0, \\
V|_{\p M} = \vp\in C(\overline{M})\cap\PSH(M,\omega),
\end{cases}
\end{equation}
where here we are now working on the compact K\"ahler manifold with pseudoconcave boundary $(M^n, \omega)$. Bremmerman \cite{Bre} showed that the solution to \eqref{Chpt2_1} (if it exists) can be expressed as a supremum of $\omega$-psh functions, an idea based on the classical Perron method for subharmonic functions:
\begin{equation}\label{Chpt2_2}
V = \sup\{ v\in \PSH(M, \omega)\ |\ v^*|_{\p M} \leq \vp\}^*,
\end{equation}
where we write $f^*$ to refer to the upper-semicontinuous regularization of $f$. The proof of this is essentially the well-known comparison principle, assuming again that a solution to \eqref{Chpt2_1} exists. The assumption that a solution exists is non-trivial -- unlike the classical case of subharmonic functions, it can fail in a number of delicate ways (see for instance the lecture notes of B\l ocki \cite{Bl_notes}). In particular, the fact that we are asking for the existence of a bounded $\omega$-psh function on $M$ is a non-trivial constraint on $\omega$, and implies it has a large degree of ``cohomological positivity."

Recently, Berman \cite{Ber} has found a new method to approximate such envelopes, in the setting of a compact K\"ahler manifold $(X^n,\omega_X)$ without boundary. His technique is inspired by the ``zero temperature limit" of statistical mechanics, and has the virtue of working well even when one is considering cohomology classes $[\alpha]$ which are only nef and big. Explicitly, he considers solutions to the following family of equations:
\[
(\alpha+ i\p\pbar u_\beta)^n = e^{\beta (u_\beta - f)}\omega_X^n
\]
as the parameter $\beta\rightarrow\infty$ (the physical interpretation is that $\beta = 1/T$, where $T$ is the temperature of the system). He then shows that this family of solutions will converge pointwise to the envelope:
\begin{equation}\label{Chpt2_3}
P_\alpha(f) = \sup\{ v\in\PSH(X,\alpha)\ |\ v\leq f\},
\end{equation}
when $f\in C^\infty(X)$. $P_\alpha(f)$ is clearly the proper analogue of \eqref{Chpt2_2} on a manifold without boundary. Further, Berman shows that this family is well-suited to {\it a priori} estimates under small perturbations of $\alpha$, which allows him to deduce a uniform Laplacian bound when $\alpha$ is nef and big. Later on, this was improved by Chu-Tosatti-Weinkove \cite{CTW} to a full Hessian bound along the path, which allows them to show $C^{1,1}$ regularity of the envelope \eqref{Chpt2_3}.

Our primary interest in this paper will be to study geodesics of K\"ahler metrics on singular varieties, defined in the sense of Moishezon \cite{Mo}. We refer the reader to \cite{Da2,DNG} for previous work in this setting (see also \cite{DNL}). Mirroring the above discussion, we will thus want to study a more general Dirichlet problem, similar to \eqref{Chpt2_1}:
\begin{equation}\label{Chpt2_4}
\begin{cases}
(\alpha + i\p\pbar V)^{n} = 0, \\
V|_{\p M} = \vp\in\PSH(M,\alpha).
\end{cases}
\end{equation}
The difference between \eqref{Chpt2_1} and \eqref{Chpt2_4} is that we do not assume that the form $\alpha$ is positive, even after allowing for the addition/subtraction of an exact form. As previously mentioned, this means that solutions to \eqref{Chpt2_4} might not exist if we assume that $\vp$ is bounded. As such, we are forced to allow $\vp$ to be potentially unbounded from below, even near $\p M$. This will clearly force any solution $V$ to also be unbounded from below -- but this is a considerable issue, as there is no known Monge-Amp\`ere operator $(\alpha + i\p\pbar V)^n$ which is always defined for unbounded $\alpha$-psh functions on $M$.

It is thus much better to start directly with the envelope definition:
\begin{equation}\label{Chpt2_5}
V = \sup\{ v\in \PSH(M, \alpha)\ |\ v^*|_{\p M} \leq \vp\}^*.
\end{equation}
If $\alpha$ is assumed to be ``big" (see Section \ref{background}), then $V$ will be locally bounded on the complement of a closed, pluripolar set $E$. It is then relatively straight forward to show that $V$ will solve \eqref{Chpt2_4} on $M\setminus E$, using the Monge-Amp\`ere operator of Bedford--Taylor \cite{BT}. Thus, it is natural to say $V$ is a ``solution" to \eqref{Chpt2_4} in this case (cf. \cite{DDL}).

Our main technical result, Theorem \ref{theorem}, establishes $C^{1,1}$-regularity for \eqref{Chpt2_5} locally on the set where $\alpha$ is strictly positive, under a number of additional assumptions. These conditions are quite natural -- the first is that the form $\alpha$ be ``nef and big" on $M$, so that we can approximate $\alpha$ by K\"ahler forms (again, see {Section} \ref{background} for terminology). The second is that $\vp$ be similarly well-approximated by smooth functions (condition (a) in Theorem \ref{theorem}). And finally, we will need to assume that $\vp$ {possess} some degree of positivity -- specifically, we will ask that $\alpha + i\p\pbar\vp$ be strictly positive on the complement of a fixed pluripolar set (which we allow to intersect $\p M$), and that this positivity be quantifiable (condition (b) in Theorem \ref{theorem}).

Similar to the smooth case, we then use our regularity result for \eqref{Chpt2_5} to show regularity of geodesics of K\"ahler metrics on compact K\"ahler varieties, as intended:

\begin{theorem}\label{cor}
Given two cohomologous K\"ahler metrics $\omega_1,\omega_2$ on a singular K\"ahler variety $X$, the geodesic connecting them is in $C^{1,1}_{\textrm{loc}}(X_\mathrm{reg}\times A)$, where $A\subset\C$ is an annulus and $X_\mathrm{reg}$ is the smooth part of $X$.
\end{theorem}

In fact, we can do even better than Theorem \ref{cor} -- Theorem \ref{theorem} is flexible enough to prove regularity of geodesics between K\"ahler currents with analytic singularities in nef and big classes:

\begin{theorem}\label{fake}
 Let $(X,\omega)$ be a smooth K\"ahler manifold without boundary, $[\alpha]$ a big and nef class, and $\vp_1,\vp_2$ two strictly $\alpha$-psh functions with analytic singularities and the same singularity type. Then the geodesic connecting $\vp_1$ and $\vp_2$ is in $C^{1,1}_{\text{\textrm{loc}}}((X\setminus \mathrm{Sing}(\vp_1))\times A)$, where again $A\subset\C$ is an annulus.
\end{theorem}

Theorem \ref{cor} was raised as an explicit question at the AIM workshops ``The complex Monge-Amp\`{e}re equation'' \cite[Question 7]{AIM2} and ``Nonlinear PDEs in real and complex geometry" \cite[Problem 2.8]{AIM}, and Theorem \ref{fake} confirms an expectation raised in \cite[pg. 396]{DDL}.  \\

\begin{remark}
Finally, we note that some of the estimates in this paper can be combined with the technique of \cite{CTW1} to improve upon the main result of \cite{He} -- the proof is straight forward but tedious, so we leave the details to the interested reader. The merit of He's technique is that it does not require any positivity of the boundary data beyond that they be quasi-psh. However, it applies only in the setting of geodesics and the overall conclusion is weaker, establishing $C^{1,1}$ regularity in the spatial directions only, i.e. not in the annular directions.
\end{remark}

The rest of the paper is organized as follows. In {Section} \ref{background}, we recall some background information and set some terminology we will use throughout the paper. Our main technical result, Theorem \ref{theorem}, is proved in Section \ref{bignef}. Our primary contribution to the existing literature is a new boundary estimate for the complex Monge-Amp\`ere operator, which is given in Proposition \ref{boundary second order estimate}. We prove Theorem \ref{cor} and Corollary \ref{cor2} in Section \ref{geometry} (Theorem \ref{fake} follows immediately from Corollary \ref{cor2}). We then briefly discuss the case of geodesic rays -- we mainly observe that the results in \cite{Mc} still apply in this generality. Finally, we include an appendix containing some estimates for the Dirichlet problem for the $\omega$-Laplacian when the boundary data is degenerating, which will be needed in the proof of Theorem \ref{theorem}. \\

\noindent
{\bf Acknowledgments. } We would like to thank Valentino Tosatti for useful discussions, and for helping improve the clarity of this paper. We would also like to thank the referee for carefully reading our manuscript, and for helpful comments and suggestions.

\section{Background and Notation}\label{background}

We begin by recalling some background material and defining some terminology we will use throughout the paper. We first set some notation. If $M$ is a compact K\"ahler manifold (possibly with boundary), and $\alpha$ is a closed, real $(1,1)$-form on $M$, then we write $\PSH(M,\alpha)$ to denote the space of all $\alpha$-psh functions, i.e. all those functions:
\[
v: M^\circ\rightarrow \R\cup\{-\infty\},
\]
such that $v$ is upper semi-continuous and $\alpha + i\p\pbar v\geq 0$ in the sense of currents on $M^\circ$. We will sometimes write $\alpha_v := \alpha + i\p\pbar v$ as shorthand. We use the notation $\mathrm{Sing}(v) := \{v = -\infty\}$ for the pole set of $v$. Also, if $f$ is a function on $M^\circ$ (or on $M$), we will write:
\[
f^*(x_0) := \limsup_{\substack{x\rightarrow x_0\\ x\not= x_0}} f(x)\ \ \forall\, x_0\in M.
\]
for the upper semicontinuous regularization of $f$. Note that $f^*$ will be a function on $M$, even if $f$ was only originally defined on $M^\circ$.

We will also deal with many functions which are exponentially smooth. We say that $\psi$ is exponentially smooth if, for any integer $k > 0$, there exists a large $C_k > 0$ such that $e^{C_k\psi}$ is $C^k$-differentiable up to the boundary of $M$. The most important example of exponentially smooth functions are those functions with analytic singularities -- recall that a function $f$ has analytic singularities if there exists $c>0$ such that, near any point, $f$ can be locally expressed as a sum:
\[
f = c\log \sum_{j=1}^N|s_j|^2 + h
\]
for some (finite) collection of holomorphic functions $s_j$ and a smooth function $h$. We have $\mathrm{Sing}(f) = \{\sum_{j=1}^N |s_j|^2 = 0\}$, so that $f$ will in general be unbounded.

Consider now a complex K\"ahler variety $(X_0^n,\omega_0)$ -- for us, this will mean that $X_0$ is a reduced, irreducible compact complex analytic space, without boundary, of complex dimension $n$, equipped with a K\"ahler metric $\omega_0$ in the sense of Moishezon \cite{Mo}. Note that we do not ask that $X_0$ be normal. Using Hironaka's theorem on resolution of singularities \cite{Hir}, there exists a proper modification:
\[
\mu: X\rightarrow X_0,
\]
such that $X$ is compact K\"ahler. It is well-known that $\mu^*\omega_0$ will be a smooth, semi-positive form on $X$, but it will never be a K\"ahler form (unless $\mu$ is trivial, of course). Given two K\"ahler metrics $\omega_0, \omega_1$ on $X_0$, we will say they are cohomologous if:
\[
[\mu^*\omega_0] = [\mu^*\omega_1]
\]
as $(1,1)$-classes on $X$ (this is just so that we only have to use Dolbeault cohomology on the smooth space $X$, instead of working on $X_0$). In this case, let $\vp_1\in\PSH(X,\mu^*\omega_0)\cap C^{\infty}(X)$ be such that:
\[
\mu^*\omega_1 = \mu^*\omega_0 + i\p\pbar\vp_1.
\]
We then define the geodesic between $\omega_0$ and $\omega_1$ to be the envelope:
\[
V:=\sup\{v\in\PSH(X\times A, \pi^*\mu^*\omega_0)\ |\  v^*|_{\{|\tau| = 1\}}\leq 0,\ v^*|_{\{|\tau| = e^{-1}\}}\leq \pi^*\vp_1\},
\]
where $A\subset \C$ is the annulus $A:= \{\tau\in\C\ |\ e^{-1}\leq |\tau|\leq 1\}$, and $\pi:X\times A\rightarrow X$ is the projection map. It is easy to check (cf. \cite{GZ_book}) that $V = V^*$, so that $V$ agrees with definition \eqref{Chpt2_5}. As $\pi^*\mu^*\omega_0$ is a semi-positive form, it follows that $V$ is bounded, and hence $V$ solves the Dirichlet problem \eqref{Chpt2_4}, using the Monge-Amp\`ere operator of Bedford-Taylor \cite{BT} (the boundary condition is guaranteed by the existence of an appropriate subsolution (cf. Section \ref{geometry})).

Let us now suppose that $(X^n,\omega)$ is a compact K\"ahler manifold, without boundary, and that $\alpha$ is a real, closed $(1,1)$-form such that $[\alpha]$ is nef and big -- recall that $[\alpha]$ is nef if $[\alpha + \e\omega]$ is a K\"ahler class for all $\e > 0$, and that $[\alpha]$ is big if there exists a K\"ahler current $\psi\in\PSH(X,\alpha)$ (i.e. $\alpha + \ddbar\psi\geq \delta\omega$, for some $\delta > 0$). By Demailly's regularization theorem, if $[\alpha]$ is big, then there exist many strictly psh functions in $\PSH(X,\alpha)$ with analytic singularities. Given two such functions, $\psi_0, \psi_1$, with the same singularity type (so that $\psi_1 - C\leq \psi_0\leq \psi_1 + C$ for some $C\in\R$), we define the geodesic between them to be the envelope:
\[
V:=\sup\{v\in\PSH(X\times A, \pi^*\alpha)\ |\  v^*|_{\{|\tau| = 1\}}\leq \pi^*\psi_0,\ v^*|_{\{|\tau| = e^{-1}\}}\leq \pi^*\psi_1\}.
\]
$V$ will  be unbounded in general, but it is again easy to check that it satisfies $(\pi^*\alpha + i\p\pbar V)^n = 0$ on $(X\setminus \mathrm{Sing}(\psi_0))\times A$ (see \cite{DDL} for example).

Finally, let $(M^n,\omega)$ be a compact K\"ahler manifold with non-empty, weakly pseudoconcave boundary. Recall that $M$ has weakly pseudoconcave boundary if the Levi form $L_{\de M,\nu}$ of $\de M$ (with respect to the outward pointing normal vector field $\nu$ on $\de M$) satisfies $L_{\de M,\nu}\leq0$. We will employ the following terminology, mirroring that used in the case when $M$ does not have boundary. We will say that a closed, real $(1,1)$-form $\alpha$ is nef if for every $\e > 0$ there exists a smooth function $u_\e\in C^{\infty}(\overline{M})$ such that $u_\e\leq 0$ and $\alpha + \e\omega + i\p\pbar u_\e > 0$. Similarly, we will say that $\alpha$ is big if there exists a function $\psi\in\PSH(M,\alpha)$ such that $\psi\leq 0$ is exponentially smooth and $\alpha + i\p\pbar\psi \geq \delta\omega$, for some $\delta > 0$. We will say that $\alpha$ is $\psi$-big if we wish to emphasize the particular choice of K\"ahler current $\psi$.

\section{$C^{1,1}$-Estimates for Big and Nef Classes}\label{bignef}

We will now present and prove our main result:

\begin{theorem}\label{theorem}
Suppose that $(M^n,\omega)$ is a compact K\"ahler manifold with weakly pseudoconcave boundary. Let $\alpha$ be a smooth, real $(1,1)$-form on $M$ that is $\psi$-big and nef, and suppose there is a function
 \[
   \vp\in\PSH(M,\alpha),\ \  \psi \leq \vp \leq 0,
 \]
 such that:
 \begin{enumerate}[(a)]
 \item There exists a sequence of smooth functions $\vp_\e\in\PSH(\alpha + \e\omega)\cap C^{\infty}(\overline{M})$, decreasing to $\vp$, such that we have the bounds:
 \[
   |\nabla \vp_\e|_\omega + |\nabla^2 \vp_\e|_\omega \leq C e^{-B_0\psi},
 \]
 for each $\e > 0$, with $B_0, C$ fixed positive constants.
 \item We have the key positivity condition:
 \begin{equation}\label{key}
  \alpha + \e\omega + \ddbar\vp_\e \geq ce^{B_0\psi}\omega,
 \end{equation}
 for each $\e > 0$, with $B_0, c$ fixed positive constants.
 \end{enumerate}
 Then the envelope:
 \[
   V := \sup\{ v\in\PSH(M,\alpha)\ |\ v^{*}|_{\partial M} \leq \vp|_{\partial M}\}
 \]
 is in $C^{1,1}_{\textrm{loc}}(M\setminus\mathrm{Sing}(\psi))$.
\end{theorem}

Theorem \ref{theorem} will be proved using the Berman path \cite{Ber} and the Hessian estimates of Chu-Tosatti-Weinkove \cite{CTW0, CTW1, CTW}. As mentioned, our main contribution is a uniform boundary Hessian estimate along the Berman path. Note that if $\vp$ in Theorem \ref{theorem} is actually smooth near the boundary, one can simply take $\vp_\e = \widetilde{\max}\{\vp, v_\e - C_\e\}$ for all $\e > 0$, where the $v_\e$ come from the nef condition, and the $C_\e$ are large constants such that $\vp_\e = \vp$ near the boundary -- this is because we won't actually need the estimates in part (a) to hold everywhere, only near the boundary, as the proof will show. In this manner we recover, and actually improve upon, \cite[Theorem 1.3]{Mc}.

The proof of Theorem \ref{theorem} is rather long, so we sketch the main steps here. We begin by simultaneously perturbing the form $\alpha$ (by adding a small multiple of $\omega$) and approximating the boundary data (using the family $\vp_\e$). Then, we run the Berman path on these perturbed data, and produce estimates uniform under both the path and the perturbations. We begin by proving uniform $L^\infty$ bounds, and then establish uniform control of the gradient. This is done by following (and simultaneously improving) similar gradient estimates in \cite{CTW}, which are based on estimates in \cite{Bl2} and \cite{PS2}. Next, we establish control of the Hessian at the boundary, adapting previous estimates from \cite{CKNS, Guan, Chen, Bou}. Finally, we bound the Hessian on the interior of $M$, using the method of \cite{CTW0}.

To begin, we first point out that one can relax the assumptions in \eqref{key} slightly -- throughout, we'll work with condition \eqref{prop of F} and this other function $F$, as it makes the proof slightly easier to follow:

\begin{proposition}\label{forms}
Suppose that $F$ is an exponentially smooth, quasi-psh function such that:
\begin{equation}\label{prop of F}
\alpha + \e\omega + \ddbar\vp_\e \geq e^{F}\omega\text{ for all }\e > 0.
\end{equation}
Then there exists an exponentially smooth, strictly $\alpha$-psh function $\ti{\psi}$ such that \eqref{key} holds for all $\e > 0$ and
\[
\alpha+\ddbar\ti\psi \geq \frac{\delta}{2}\omega.
\]
Further, $\ti{\psi}$ will be singular only along $\mathrm{Sing}(\psi)\cup \mathrm{Sing}(F)$.
\end{proposition}
\begin{proof}
Let $\psi$ be as in the big condition. We may assume without loss of generality that $F\leq 0$. By assumption,
\[
\ddbar F \geq -C\omega,
\]
so if we define:
\[
\ti{\psi} := \psi + \frac{\delta}{2C} F,
\]
we have that \eqref{key} holds with $c=1$ and $B_0 = 2C/\delta$.
\end{proof}

Note that, if we happen to have $\vp_\e = \vp$ for all $\e > 0$ (so that $\vp$ is actually smooth and $\alpha$ is ``semipositive''), then we can scale $\omega$ such that $\omega \geq \alpha_\vp$. Then the function
\[
F := \log\left((\alpha_\vp)^n/\omega^n\right)
\]
 is exponentially smooth and satisfies \eqref{prop of F} -- to see this, look at the eigenvalues $\lambda _j$ of $\alpha_\vp$ in normal coordinates for $\omega$ at a point:
\[
 \lambda_j = \left(\prod_{k=1}^n \lambda_k\right) / \left(\prod_{\substack{k=1\\ k\not= j}}^n\lambda_k\right) \geq e^F.
\]
Unfortunately, such an $F$ will not always be quasi-psh -- however, we will show in Section \ref{geometry} that in the case of a geodesic between two K\"ahler metrics on a singular K\"ahler variety, we can always find an exponentially smooth $F'\leq F$ that will be quasi-psh, and hence we will be able to apply our results in that setting.

Before moving on, observe that being exponentially smooth gives control over all derivatives of $F$ and $\psi$:
\begin{equation}\label{derivatives}
  \begin{split}
  |\nabla\psi|_{g}+|\nabla^{2}\psi|_{g} \leq C e^{-B_0\psi}\\
  |\nabla F|_{g}+|\nabla^{2}F|_{g} \leq C e^{-B_0F},
 \end{split}
\end{equation}
where here $g$ is the Riemannian metric corresponding to $\omega$. Also, by replacing $\e$ with $\e/2$ and relabeling the $\vp_\e$, we can improve condition \eqref{prop of F} to
\begin{equation}\label{key'}
 \alpha + \e\omega + \ddbar\vp_\e \geq (e^F + \e/2)\omega,
\end{equation}
without changing the problem. Finally, we will also assume without loss of generality that:
\begin{equation}\label{WLOG}
 \alpha \leq \omega.
\end{equation}

\bigskip

\begin{proof}[Proof of Theorem \ref{theorem}]

Our strategy will be the same as in \cite{Mc}, which is a combination of the techniques in \cite{Ber} and the estimates in \cite{CTW}. We begin by approximating $V$ by the following envelopes:
\[
 V_\e := \sup\{ v\in\PSH(M,\alpha + \e\omega)\ |\ v^{*}|_{\partial M} \leq \vp_\e|_{\partial M}\}.
\]
Note that the $V_\e$ decrease pointwise to $V$ as $\e$ decreases to $0$.

We now define the obstacle functions $h_\e \in C^{\infty}(M)$ as the solutions to the following Dirichlet problems:
\[
 \begin{cases}
  \Delta_{2\omega} h_\e = -n, \\
  h_\e|_{\partial M} = \vp_\e|_{\partial M}.
 \end{cases}
\]
By Lemma \ref{applem} and \ref{applem 1}, we have control over the derivatives of the $h_\e$:
\begin{equation}\label{estimates for h ve}
 |\nabla h_\e|_g + |\nabla^2 h_\e|_g + |\nabla^3 h_\e|_g \leq C e^{-B_0 \psi},
\end{equation}
and we know that they decrease as $\e\rightarrow 0$. In particular, for all $\e \leq 1$, $h_\e \leq h_1 \leq C$, independent of $\e$.

Observe that for any $v\in\PSH(M,\alpha + \e\omega)$ with $v^{*}|_{\partial M} \leq \vp_\e|_{\partial M}$, we have that:
\begin{equation}\label{upper bound 1}
 v \leq h_\e
\end{equation}
by (\ref{WLOG}) and the weak maximum principle for the $\omega$-Laplacian. We thus see that:
\[
 V_\e = \sup\{ v\in\PSH(M,\alpha + \e\omega)\ |\ v\leq h_\e\},
\]
where the inequality now holds on all of $M$.

We now approximate the $V_\e$ by the smooth solutions to the non-degenerate Dirichlet problem:
\begin{equation}\label{ma}
\begin{cases}
 (\alpha + \ve\omega + \ddbar u_{\e,\beta})^n = e^{\beta (u_{\e,\beta} - h_{\ve}) + n\log(\ve/4)}\omega^n, \\
 \alpha + \ve\omega + \ddbar u_{\e,\beta} > 0, \ \ u_{\e,\beta}|_{\partial M} = \vp_\e|_{\partial M}.
\end{cases}
\end{equation}
By \cite[Proposition 2.3]{Ber}, \cite[Proposition 4.5]{Mc}, the $u_{\e,\beta}$ converge uniformly to $V_\e$ as $\beta\rightarrow\infty$ (the convergence is not uniform in $\e$, but this will not be a problem as we will send $\beta$ to infinity before sending $\e$ to zero). Note that as $u_{\e,\beta}\in\PSH(M,\alpha+\e\omega)$ with $u_{\ve,\beta}|_{\partial M} = \vp_\e|_{\partial M}$, by (\ref{upper bound 1}), we have:
\begin{equation}\label{upper bound}
 u_{\e,\beta} - h_\e \leq 0.
\end{equation}
Also note that:
\[
 (\alpha + \ve\omega + \ddbar \vp_\e)^n \geq (\e/2)^n\omega^n \geq e^{\beta (\vp_\e - h_\e) + n\log(\e/4)}\omega^n
\]
where we used \eqref{key'} and $\vp_\e\leq h_\e$. This makes $\vp_\e$ a subsolution of \eqref{ma}, so that $u_{\e,\beta}$ actually exists (cf. \cite[Theorem 1.1]{CKNS}, \cite[Theorem 1.1]{Guan}, \cite[Theorem 1.1]{GL}), and we have
\begin{equation}\label{lower bound}
\vp_\e \leq u_{\e,\beta}.
\end{equation}

Our goal is to establish uniform $C^2$-estimates for $u_{\ve,\beta}$, independent of $\ve$ and $\beta$ (we drop the subscripts now for ease of notation). We have just shown the requisite $C^0$-bound:
\begin{equation}\label{zero order estimate}
 \psi \leq \vp \leq \vp_\e\leq u \leq h_\e \leq h_1
\end{equation}
and since $\vp_\e|_{\partial M} = u|_{\partial M} = h_\e|_{\partial M}$, it follows that the gradient is bounded on $\partial M$:
\begin{equation}\label{gradbound}
 |\nabla u|_g \leq |\nabla \vp_\e|_g + |\nabla h_\e|_g \leq Ce^{-B_0\psi}\text{ on }\partial M.
\end{equation}

We now bound the gradient on the interior by following the argument in \cite[Lemma 4.1, (iii)]{CTW}:
\begin{lemma}\label{first order estimate}
There exist uniform constants $\beta_0, B$, and $C$ such that
\[
  |\nabla u|_{g} \leq Ce^{-B\psi}\text{ for all }\beta\geq \beta_0.
\]
\end{lemma}

\begin{proof}

We begin by defining:
\[
 \ti{\omega} := \alpha + \e\omega + \ddbar u.
\]
Then note that, as $u\geq\psi$ and $\psi\leq 0$, we have:
\[
 \ti{u} := u  - (1+\delta)\psi \geq -\delta\psi \geq 0,
\]
and:
\[
 \ti{\omega} - \ddbar\ti{u}
  =  (1+\delta)(\alpha+\ddbar\psi)-\delta\alpha+\ve\omega
\geq (1+\delta)\delta\omega-\delta\alpha
\geq \delta^2\omega,
\]
by \eqref{WLOG}. In particular, we have
\begin{equation}\label{C1 eqn 1}
\Delta_{\ti{g}}\ti{u}\leq n-\delta^{2}\tr{\ti{\omega}}{\omega}.
\end{equation}

We now seek to bound the quantity:
\[
 Q := e^{H(\ti{u})}|\nabla u|_g^2,
\]
by a constant independent of $\e$ and $\beta$, where here $H(s)$ is defined for $s\geq 0$ as:
\[
 H(s) := -Bs + \frac{1}{s + 1},
\]
for some large constant $B$ to be determined. Let $x_0$ be a maximum point for $Q$ -- it cannot be in $\mathrm{Sing}(\psi)$, as $Q$ is zero there. If it is on the boundary of $M$, by \eqref{gradbound} and $\ti{u}\geq-\delta\psi$, we have
\begin{equation*}
Q(x_{0}) \leq Ce^{-B\ti{u}(x_{0})-B_{0}\psi(x_{0})} \leq Ce^{(B\delta-B_{0})\psi(x_{0})}\leq C,
\end{equation*}
provided we take $B > B_0/\delta$.

Thus, suppose that $x_0$ is an interior point. It suffices to prove that
\begin{equation}\label{goal}
|\nabla u|_{g}^{2}(x_{0}) \leq Ce^{-B_{0}\psi(x_{0})},
\end{equation}
for some uniform constant $C$ -- by the same argument as above, this would imply $Q$ is uniformly bounded.

Next, choose holomorphic normal coordinates such that
\[
g_{i\ov{j}} = \delta_{ij}, \quad
\ti{g}_{i\ov{j}} = \delta_{ij}\ti{g}_{i\ov{i}}, \quad \text{at $x_0$},
\]
where $\ti{g}$ is the Riemannian metric corresponding to $\ti{\omega}$. At $x_0$, we see that
\begin{equation}\label{C1 eqn 2}
\begin{split}
0 \geq  \Delta_{\ti{g}} Q
=  |\nabla u|_{g}^{2}\Delta_{\ti{g}}(e^{H})+e^{H}\Delta_{\ti{g}}(|\nabla u|_{g}^{2})+2\mathrm{Re}\Big(\ti{g}^{i\ov{i}}(e^{H})_{i}(|\nabla u|_{g}^{2})_{\ov{i}}\Big).
\end{split}
\end{equation}
For the first term of \eqref{C1 eqn 2}, using \eqref{C1 eqn 1} and $H'\leq0$, we have
\begin{equation}\label{C1 eqn 3}
\begin{split}
|\nabla u|_{g}^{2}&\Delta_{\ti{g}}(e^{H}) \\
= {} & e^{H}\left((H')^2+H''\right)|\nabla u|_{g}^{2}|\nabla \ti{u}|_{\ti{g}}^{2}+e^{H}H'|\nabla u|_{g}^{2}\Delta_{\ti{g}}\ti{u} \\
\geq {} & e^{H}\left((H')^2+H''\right)|\nabla u|_{g}^{2}|\nabla \ti{u}|_{\ti{g}}^{2}+ne^{H}H'|\nabla u|_{g}^{2}-\delta^{2}e^{H}H'|\nabla u|_{g}^{2}\tr{\ti{\omega}}{\omega}.
\end{split}
\end{equation}
For the second term of \eqref{C1 eqn 2}, applying $\de_{k}$ to \eqref{ma} gives:
\[
\ti{g}^{i\ov{i}}u_{i\ov{i}k} = \beta(u-h_{\ve})_{k}-\ti{g}^{i\ov{i}}\de_{k}\alpha_{i\ov{i}}.
\]
Thus,
\begin{equation}\label{C1 eqn 4}
\begin{split}
e^{H}\Delta_{\ti{g}}(|\nabla u|_{g}^{2})
= {} & e^{H}\sum_{k}\ti{g}^{i\ov{i}}(|u_{ik}|^{2}+|u_{i\ov{k}}|^{2})+2e^{H}\mathrm{Re}\bigg(\sum_{k}\ti{g}^{i\ov{i}}u_{i\ov{i}k}u_{\ov{k}}\bigg) \\
& +e^{H}\ti{g}^{i\ov{i}}(\de_{i}\de_{\ov{i}}g^{k\ov{l}})u_{k}u_{\ov{l}} \\[2mm]
\geq {} & e^{H}\sum_{k}\ti{g}^{i\ov{i}}(|u_{ik}|^{2}+|u_{i\ov{k}}|^{2})+2\beta e^{H}|\nabla u|_{g}^{2} \\
& -2\beta e^{H}\mathrm{Re}\langle\nabla h_\e,\overline{\nabla}u\rangle_g-Ce^{H}|\nabla u|_{g}^{2}\tr{\ti{\omega}}{\omega}.
\end{split}
\end{equation}
For the third term of \eqref{C1 eqn 2}, we compute
\[
2\mathrm{Re}\left(\ti{g}^{i\ov{i}}(e^{H})_{i}(|\nabla u|_{g}^{2})_{\ov{i}}\right)
= 2e^{H}H'\mathrm{Re}\left(\sum_{k}\ti{g}^{i\ov{i}}\ti{u}_{i}u_{k\ov{i}}u_{\ov{k}}\right)
+2e^{H}H'\mathrm{Re}\left(\sum_{k}\ti{g}^{i\ov{i}}\ti{u}_{i}u_{k}u_{\ov{k}\ov{i}}\right).
\]
Recalling $u_{k\ov{i}}=\ti{g}_{k\ov{i}}-\alpha_{k\ov{i}}-\ve g_{k\ov{i}}$, $\ti{u}=u-(1+\delta)\psi$ and $H'\leq0$, and using the Cauchy-Schwarz inequality,
\[
\begin{split}
2e^{H}&H'\mathrm{Re}\bigg(\sum_{k}\ti{g}^{i\ov{i}}\ti{u}_{i}u_{k\ov{i}}u_{\ov{k}}\bigg) \\
= {} & 2e^{H}H'\mathrm{Re}\langle\nabla\ti{u},\overline{\nabla}u\rangle_g
-2e^{H}H'\mathrm{Re}\bigg(\sum_{k}\ti{g}^{i\ov{i}}\ti{u}_{i}(\alpha_{k\ov{i}}+\ve g_{k\ov{i}})u_{\ov{k}}\bigg) \\
\geq {} & 2e^{H}H'\mathrm{Re}\langle\nabla\ti{u},\overline{\nabla}u\rangle_g+\frac{\delta^{2}}{2}e^{H}H'|\nabla u|_{g}^{2}\tr{\ti{\omega}}{\omega}
+\frac{C}{\delta^{2}}e^{H}H'|\nabla \ti{u}|_{\ti{g}}^{2} \\
\geq {} & 2e^{H}H'|\nabla u|_{g}^{2}-2(1+\delta)e^{H}H'\mathrm{Re}\langle\nabla\psi,\overline{\nabla}u\rangle_g+\frac{\delta^{2}}{2}e^{H}H'|\nabla u|_{g}^{2}\tr{\ti{\omega}}{\omega}
+\frac{C}{\delta^{2}}e^{H}H'|\nabla \ti{u}|_{\ti{g}}^{2} \\
\geq {} & 3e^{H}H'|\nabla u|_{g}^{2}+Ce^{H}H'|\nabla\psi|_{g}^{2}+\frac{\delta^{2}}{2}e^{H}H'|\nabla u|_{g}^{2}\tr{\ti{\omega}}{\omega}
+\frac{C}{\delta^{2}}e^{H}H'|\nabla \ti{u}|_{\ti{g}}^{2}
\end{split}
\]
and
\[
2e^{H}H'\mathrm{Re}\left(\sum_{k}\ti{g}^{i\ov{i}}\ti{u}_{i}u_{k}u_{\ov{k}\ov{i}}\right)
\geq -e^{H}\sum_{k}\ti{g}^{i\ov{i}}|u_{ik}|^{2}-e^{H}(H')^{2}|\nabla u|_{g}^{2}|\nabla\ti{u}|_{\ti{g}}^{2}.
\]
It then follows that
\begin{equation}\label{C1 eqn 5}
\begin{split}
2\mathrm{Re}&\left(\ti{g}^{i\ov{i}}(e^{H})_{i}(|\nabla u|_{g}^{2})_{\ov{i}}\right) \\
\geq {} & 3e^{H}H'|\nabla u|_{g}^{2}+Ce^{H}H'e^{-B_{0}\psi}+\frac{\delta^{2}}{2}e^{H}H'|\nabla u|_{g}^{2}\tr{\ti{\omega}}{\omega} \\
& +\frac{C}{\delta^{2}}e^{H}H'|\nabla \ti{u}|_{\ti{g}}^{2}-e^{H}\sum_{k}\ti{g}^{i\ov{i}}|u_{ik}|^{2}-e^{H}(H')^{2}|\nabla u|_{g}^{2}|\nabla\ti{u}|_{\ti{g}}^{2},
\end{split}
\end{equation}
where we used $|\nabla\psi|_{g}^{2}\leq Ce^{-B_{0}\psi}$. Substituting \eqref{C1 eqn 3}, \eqref{C1 eqn 4} and \eqref{C1 eqn 5} into \eqref{C1 eqn 2}, we obtain
\begin{align*}
0 \geq {} & H'' |\nabla u|_g^2|\nabla \ti{u}|_{\ti{g}}^2 + \left(-\frac{\delta^2}{2}H' - C\right)|\nabla u|^2_g\tr{\ti{\omega}}{\omega}\\
  &+ (CH' + 2\beta)|\nabla u|^2_g - 2\beta\mathrm{Re}\langle\nabla h_\e,\overline{\nabla}u\rangle_g+ CH'e^{-B_0\psi} + CH'|\nabla\ti{u}|_{\ti{g}}^2,
\end{align*}
for some uniform constant $C > 1$. Picking then
\[
 B = \max\{(2/\delta^2)(C+1), B_0/\delta, 3\},
 \]
we may use the definition of $H$ to see that for $\beta \geq \beta_0 := C(B+1) + 1$, we have:
\begin{align}\label{thing}
 0\geq \frac{2|\nabla u|_g^2|\nabla \ti{u}|_{\ti{g}}^2}{(\ti{u} + 1)^3} + (\beta&+1)|\nabla u|^2_g - 2\beta|\nabla h_\e|_g|\nabla u|_g \\&- C(B+1)e^{-B_0\psi} - C(B+1)|\nabla\ti{u}|_{\ti{g}}^2. \nonumber
\end{align}
We may now assume that at $x_0$ we have both:
\[
 |\nabla u|_g^2 \geq C(B+1)(\ti{u}+1)^3
\]
and
\[
|\nabla u|_g \geq 2|\nabla h_\e|_g.
\]
If either condition fails, we obtain $|\nabla u|_{g}^{2}(x_{0}) \leq Ce^{-B_{0}\psi(x_{0})}$ directly. Otherwise, we still get:
\[
 C(B+1)e^{-B_0\psi} \geq |\nabla u|_g^2(x_0),
\]
from \eqref{thing}, as required.
\end{proof}

We now bound the Hessian near the boundary.

\begin{proposition}\label{boundary second order estimate}
Assume that the key condition \eqref{key'} is satisfied for each $\e > 0$. Then there exist uniform constants $B$ and $C$ such that
\[
|\nabla^{2}u|_{g} \leq Ce^{-B(F + \psi)} ~\text{~on $\de M$}.
\]
\end{proposition}

\begin{proof}
Fix a point $p\in\partial M$, and center coordinates $(\{z_{i}\}_{i=1}^{n}, B_R)$ at $p$, where here $B_R$ is a ball of radius $R$. Write $z_{i}=x_{2i-1}+\sqrt{-1}x_{2i}$ for $1\leq i\leq n$. Let $r$ be a defining function for $M$ in $B_R$ (so that $\{r \leq 0\} = M\cap B_R$ and $\{r = 0\} = \partial M\cap B_R$). After making a linear change of coordinates, we may assume that at $p$
\[
r_{x_{2n}} = -1  \quad \text{and} \quad
r_{x_{\alpha}} = 0 \quad \text{for $1\leq\alpha\leq 2n-1$}.
\]
We then have that the Taylor expansion for $r$ near $p$ is of the form:
\[
r = \mathrm{Re}\left(\sqrt{-1}z_{n}+\sum_{i,j=1}^{n}b_{ij}z_{i}z_{j}\right)
+\sum_{i,j=1}^{n}a_{i\ov{j}}z_{i}\ov{z}_{j}+O(|z|^{3}) \quad \text{near $p$}.
\]
Thus, if we consider the quadratic change of coordinates
\[
z_{i}'=
\begin{cases}
z_{i} \ & \mbox{if $1\leq i\leq n-1$},\\
z_{n}-\sqrt{-1}\sum_{i,j=1}^{n}b_{ij}z_{i}z_{j} \ & \mbox{if $i=n$},
\end{cases}
\]
we see that
\begin{equation*}
  r(z) = -x_{2n}' + \sum_{i,j=1}^{n}a_{i\ov{j}}z_{i}'\ov{z}_{j}' + O(|z'|^{3}) ~\text{~near $p$}.
\end{equation*}
For convenience, we will use $z$ to again denote this new coordinate system, so that
\begin{equation*}
  r(z) = -x_{2n} + \sum_{i,j=1}^{n}a_{i\ov{j}}z_{i}\ov{z}_{j} + O(|z|^{3}) ~\text{~near $p$}.
\end{equation*}
In particular, after shrinking $R$, we can assume that,
\begin{equation}\label{derivative of r}
|r_{x_{2n}}+1| \leq \frac{1}{10} \ \ \text{and} \ \ |r_{\gamma}|\leq C|z|~\text{~for $1\leq\gamma\leq2n-1$},
\end{equation}
on $B_R$, for some uniform constant $C$.

As in \cite[pg. 272]{Bou}, we now define the tangent vector fields
\begin{equation*}
  D_{\gamma} = \frac{\de}{\de x_{\gamma}} -\frac{r_{x_{\gamma}}}{r_{x_{2n}}}\frac{\de}{\de x_{2n}} ~\text{~for $1\leq\gamma\leq2n-1$}
\end{equation*}
and the normal vector field
\begin{equation*}
  D_{2n} = -\frac{1}{r_{x_{2n}}}\frac{\de}{\de x_{2n}}.
\end{equation*}
Recall that we are writing:
\[
\ti{\omega} := \alpha+\e\omega+\ddbar u
\]
and
\begin{equation*}
  \ti{\omega} = \sqrt{-1}\sum_{i,j=1}^{n}\ti{g}_{i\ov{j}}dz_{i}\wedge d\ov{z}_{j}.
\end{equation*}
We also denote the inverse matrix of $(\ti{g}_{i\ov{j}})$ by $(\ti{g}^{i\ov{j}})$. Throughout, $C$ will be a constant, independent of $\e,\beta$, whose exact value may change from line to line.

We split the proof into three steps:

\bigskip
\noindent
{\bf Step 1.} The tangent-tangent derivatives.

\bigskip

Since $u=\vp_\e$ on $\de M$, at $0$ ($p\in\de M$), we have
\begin{equation*}
  |D_{\gamma}D_{\eta}u| = |D_{\gamma}D_{\eta}\vp_\e| \leq Ce^{-C\psi}  ~\text{~for $1\leq\gamma,\eta\leq 2n-1$},
\end{equation*}
as desired.

\bigskip
\noindent
{\bf Step 2.} The tangent-normal derivatives.

\bigskip

We define
\[
 U := u - h_\e, \quad \de_{\gamma}U := \frac{\de U}{\de x_{\gamma}}.
\]
For $1\leq\gamma\leq2n-1$, recalling the definition of $D_{\gamma}$, we have
\[
D_{\gamma}U =  \de_{\gamma}U-\frac{r_{x_{\gamma}}}{r_{x_{2n}}}\de_{2n}U.
\]
On $B_{R}\cap\de M$, since $U=0$, then we have $D_{\gamma}U=0$. Combining this with Lemma \ref{first order estimate} and \eqref{derivative of r}, for $1\leq\gamma\leq2n-1$, we have
\begin{equation}\label{de gamma U}
|\de_{\gamma}U| = \left|\frac{r_{x_{\gamma}}}{r_{x_{2n}}}\de_{2n}U\right| \leq Ce^{-B_{0}\psi}|z| \ \ \text{on $B_{R}\cap\de M$}.
\end{equation}
For notation, we will use $\gamma$ to denote real directions, and $i,j$ to denote complex directions. We will also need to single out the (real) $x_{2n-1}$ direction, so we will write
\[
\ti{\gamma} := 2n-1
\]
for convenience.

We begin with the exceptional $\ti{\gamma}$-direction. We will need to consider the following barrier function:
\[
\xi := \mu_1 (u - \vp_\e) + e^{B F}\left(\mu_2|z|^2- e^{B\psi}(\de_{\ti{\gamma}}U)^2\right),
\]
where $\mu_{1}$, $\mu_{2}$, and $B$ are large, uniform, constants to be determined such that:
\[
\xi\geq 0\text{ on } B_R\cap M.
\]
We will do this by a minimum principle argument. First, we show that $\xi$ is non-negative on the boundary of $B_R\cap M$. $\partial(B_R\cap M)$ has two components, $\partial B_R\cap M$ and $B_R\cap \partial M$. Recalling that $u\geq\vp_\e$, on $\partial B_R\cap M$, using Lemma \ref{first order estimate}, we see that
\begin{equation*}
\xi \geq e^{B F}\left(\mu_2|z|^2- e^{B\psi}(\de_{\ti{\gamma}}U)^2\right)
\geq e^{BF}\left(\mu_{2}R^{2}-Ce^{(B-2B_{0})\psi}\right).
\end{equation*}
On $B_R\cap \partial M$, using \eqref{de gamma U}, we see that
\[
\xi = e^{B F}\left(\mu_2|z|^2- e^{B\psi}(\de_{\ti{\gamma}}U)^2\right)
\geq e^{B F}\left(\mu_2|z|^2-Ce^{(B-2B_{0})\psi}|z|^2\right).
\]
Hence, for any choice of $B$ and $\mu_2$ sufficiently large, we can arrange that
\begin{equation}\label{tangent-normal equ 3}
\xi\geq 0 ~\text{on $\partial (B_R\cap M)$}.
\end{equation}
We will fix such a $\mu_2$ now once and for all, reserving the ability to increase $B$ as needed.

Suppose then that $x_0$ is a minimum point of $\xi$. If $x_0\in\p(B_R\cap M)$, then we are done, so suppose that $x_0$ is an interior minimum of $\xi$. Since $\psi, F$ are exponentially smooth, we have that $\xi$ is at least $C^2$ for all sufficiently large $B$, so that the minimum principle applies.

We thus wish to compute:
\begin{equation}\label{xi_equation}
\begin{split}
  \Delta_{\ti{g}}\xi = {} & \mu_{1}\Delta_{\ti{g}}(u-\vp_\e) + \mu_2\Delta_{\ti{g}}(e^{BF}|z|^2)-\Delta_{\ti{g}}\left(e^{B(F+\psi)}(\de_{\ti{\gamma}}U)^2\right).
\end{split}
\end{equation}
For the first term, we observe that:
\[
\Delta_{\ti{g}} (u - \vp_\e) = n - \tr{\ti{\omega}}{(\alpha+\ve\omega + \ddbar\vp_\e)} \leq n - (e^F + \e/2)\tr{\ti\omega}{\omega},
\]
by \eqref{key'}, and that
\begin{equation*}
 \frac{\ve}{2}\tr{\ti{\omega}}{\omega} \geq \frac{n \ve}{2}\left(\frac{\omega^n}{(\alpha +\e\omega + \ddbar u)^n}\right)^{1/n} = \frac{\ve}{2} ne^{-(\beta/n)(u - h_{\ve}) - \log(\ve/4)},
\end{equation*}
by the arithmetic-geometric mean inequality and (\ref{ma}). As $u-h_\e\leq 0$, we have then that
\begin{equation}\label{lower bound of trace}
\frac{\ve}{2}\tr{\ti{\omega}}{\omega} \geq \frac{\ve}{2} ne^{- \log(\ve/4)} = 2n.
\end{equation}
Thus,
\begin{equation}\label{tangent-normal first term}
 \mu_{1}\Delta_{\ti{g}} (u - \vp_\e) \leq -\mu_1(e^F + \e/4)\tr{\ti{\omega}}{\omega}.
\end{equation}
For the second term of \eqref{xi_equation}, by \eqref{derivatives} and a direct calculation, we obtain
\begin{equation}\label{tangent-normal second term}
  \mu_2\Delta_{\ti{g}}(e^{BF}|z|^2) \leq CB^{2}e^{(B-B_0)F}\tr{\ti{\omega}}{\omega}.
\end{equation}
For the third term of (\ref{xi_equation}), we have:
\begin{equation}\label{xi_one}
\begin{split}
 -\Delta_{\ti{g}}&\left(e^{B(F+\psi)}(\de_{\ti{\gamma}}U)^2\right) \\[1mm]
= {} & -(\de_{\ti\gamma}U)^{2}\Delta_{\ti{g}}(e^{B(F+\psi)})-e^{B(F+\psi)}\Delta_{\ti{g}}\left((\de_{\ti{\gamma}}U)^2\right) \\
& -2Be^{B(F+\psi)}\ti{g}^{i\ov{j}}\left((F+\psi)_{i}(\de_{\ov{j}}\de_{\ti{\gamma}}U)(\de_{\ti{\gamma}}U)
+(F+\psi)_{\ov{j}}(\de_{i}\de_{\ti{\gamma}}U)(\de_{\ti{\gamma}}U)\right) \\
\leq {} & CB^2e^{(B-B_0)(F+\psi)}\tr{\ti{\omega}}\omega-e^{B(F+\psi)}\Delta_{\ti{g}}\left((\de_{\ti{\gamma}}U)^2\right) \\
& +e^{B(F+\psi)}\ti{g}^{i\ov{j}}(\de_{i}\de_{\ti{\gamma}}U)(\de_{\ov{j}}\de_{\ti{\gamma}}U)
+4B^{2}e^{B(F+\psi)}(\de_{\ti{\gamma}}U)^{2}\ti{g}^{i\ov{j}}(F+\psi)_{i}(F+\psi)_{\ov{j}} \\
\leq {} & CB^{2}e^{(B-B_0)(F+\psi)}\tr{\ti{\omega}}\omega+e^{B(F+\psi)}\ti{g}^{i\ov{j}}(\de_{i}\de_{\ti{\gamma}}U)(\de_{\ov{j}}\de_{\ti{\gamma}}U)
-e^{B(F+\psi)}\Delta_{\ti{g}}\left((\de_{\ti{\gamma}}U)^2\right).
\end{split}
\end{equation}
We seek to bound
\begin{equation}\label{xi_two}
-\Delta_{\ti{g}}\left((\de_{\ti{\gamma}}U)^2\right)
= -2(\de_{\ti{\gamma}}U)\Delta_{\ti{g}}(\de_{\ti{\gamma}}U)-2\ti{g}^{i\ov{j}}(\de_{i}\de_{\ti{\gamma}}U)(\de_{\ov{j}}\de_{\ti{\gamma}}U).
\end{equation}
Applying $\de_{{\ti\gamma}}$ to (\ref{ma}) gives,
\begin{equation*}
  \tr{\ti{\omega}}{(\de_{{\ti\gamma}}(\alpha + {\ve\omega})+\de_{{\ti\gamma}}\ddbar u)} = \beta(\de_{\ti\gamma}U) + \tr{\omega}{(\de_{\ti\gamma} \omega)},
\end{equation*}
where we are letting $\de_{\ti\gamma}$ act on the components of $\alpha$ and $\omega$ in the (fixed) $z$-coordinates. Hence,
\begin{equation*}
\Delta_{\ti{g}}(\de_{{\ti\gamma}}u) = \tr{\ti{\omega}}{(\p_{\ti\gamma} \sqrt{-1}\p\pbar u)}
= \beta(\de_{\ti\gamma}U) + \tr{\omega}{(\de_{\ti\gamma} \omega)} - \tr{\ti{\omega}}(\de_{{\ti\gamma}}(\alpha + {\ve\omega}))
\end{equation*}
and
\[
\begin{split}
\Delta_{\ti{g}}(\de_{{\ti\gamma}}U)
= {} & \Delta_{\ti{g}}(\de_{{\ti\gamma}}u)-\Delta_{\ti{g}}(\de_{{\ti\gamma}}h_{\ve})\\
= {} & \beta(\de_{\ti\gamma}U) + \tr{\omega}{(\de_{\ti\gamma} \omega)} - \tr{\ti{\omega}}(\de_{{\ti\gamma}}(\alpha + {\ve\omega}))
-\Delta_{\ti{g}}(\de_{{\ti\gamma}}h_{\ve}).
\end{split}
\]
It then follows that
\begin{equation*}
-2(\de_{\ti{\gamma}}U)\Delta_{\ti{g}}(\de_{\ti{\gamma}}U) \leq  -2\beta(\de_{\ti\gamma}U)^{2}  + Ce^{-B_0\psi}\tr{\ti{\omega}}{\omega},
\end{equation*}
by Lemma \ref{first order estimate}, \eqref{estimates for h ve}, and \eqref{lower bound of trace}. Plugging this back into \eqref{xi_two} gives:
\begin{equation*}
-\Delta_{\ti{g}}\left((\de_{\ti{\gamma}}U)^2\right)
\leq -2\beta(\de_{\ti\gamma}U)^{2} + Ce^{-B_0\psi}\tr{\ti{\omega}}{\omega}
-2\ti{g}^{i\ov{j}}(\de_{i}\de_{\ti{\gamma}}U)(\de_{\ov{j}}\de_{\ti{\gamma}}U).
\end{equation*}
Substituting into \eqref{xi_one}, we get:
\begin{equation}\label{xi_three}
\begin{split}
-\Delta_{\ti{g}}(e^{B(F+\psi)}(\de_{\ti{\gamma}}U)^{2}) &\leq  CB^{2}e^{(B-B_0)(F+\psi)}\tr{\ti{\omega}}\omega \\
& -e^{B(F+\psi)}\ti{g}^{i\ov{j}}(\de_{i}\de_{\ti{\gamma}}U)(\de_{\ov{j}}\de_{\ti{\gamma}}U)-2\beta e^{B(F+\psi)}(\de_{\ti\gamma}U)^{2}.
\end{split}
\end{equation}
We can now combine \eqref{tangent-normal first term}, \eqref{tangent-normal second term}, and \eqref{xi_three} to get
\begin{equation*}
\begin{split}
\Delta_{\ti{g}}\xi \leq {} & -\mu_1(e^F + \e/4)\tr{\ti{\omega}}{\omega} + CB^{2}e^{(B-B_0)(F+\psi)}\tr{\ti{\omega}}{\omega} \\[2mm]
& -e^{B(F+\psi)}\ti{g}^{i\ov{j}}(\de_{i}\de_{\ti{\gamma}}U)(\de_{\ov{j}}\de_{\ti{\gamma}}U)-2\beta e^{B(F+\psi)}(\de_{\ti\gamma}U)^{2} \\[2mm]
\leq {} & -\left(\mu_1(e^F + \e/4)-CB^{2}e^{(B-B_0)(F+\psi)}\right)\tr{\ti{\omega}}{\omega} \\[1mm]
& -e^{B(F+\psi)}\ti{g}^{i\ov{j}}(\de_{i}\de_{\ti{\gamma}}U)(\de_{\ov{j}}\de_{\ti{\gamma}}U).
\end{split}
\end{equation*}
Thus, choosing $B$, $\mu_{1}$ sufficiently large, it then follows that
\begin{equation}\label{Delta xi}
\Delta_{\ti{g}}\xi < -e^{B(F+\psi)}\ti{g}^{i\ov{j}}(\de_{i}\de_{\ti{\gamma}}U)(\de_{\ov{j}}\de_{\ti{\gamma}}U).
\end{equation}
But since $x_{0}$ was assumed to be an interior minimum, we must have:
\[
\Delta_{\ti{g}}\xi(x_0) \geq 0,
\]
which is a contradiction. Hence, there are no interior minima for $\xi$, so that $\xi\geq0$ on $B_{R}\cap M$.

\bigskip

We will now use $\xi$ to bound the tangent-normal derivatives. Fix $1\leq \gamma\leq 2n-1$. We consider the following quantity:
\begin{equation*}
  w = \mu_1' (u - \vp_\e) + e^{B F}\left(\mu_2'|z|^2 - e^{B\psi}|D_\gamma U| - e^{B\psi}(D_\gamma U)^2\right) + \xi,
\end{equation*}
where $\mu_{1}'$, $\mu_{2}'$ are large, uniform, constants to be determined later and $B$ is the constant in the definition of $\xi$ (which we can still increase freely, up to increasing $\mu_1$ accordingly).

We now claim that $\mu_1'$ and $\mu_2'$ can be chosen such that
\begin{equation}\label{tangent-normal claim}
  w \geq 0 ~\text{on $B_R\cap M$}.
\end{equation}
Given the claim, we can control the tangent-normal derivatives as follows. Unwinding the definition of $\xi$, we see that:
\begin{equation*}
\begin{split}
w = {} & (\mu_1 + \mu'_1) (u - \vp_\e) \\
& + e^{B F}\left((\mu_2 + \mu_2')|z|^2 - e^{B\psi}|D_\gamma U| - e^{B\psi}(D_\gamma U)^2 - e^{B\psi}(\de_{\ti{\gamma}}U)^2\right).
\end{split}
\end{equation*}
Dropping the square terms, we get:
\begin{equation*}
  |D_\gamma U| \leq (\mu_1+\mu_1') e^{-B(F+\psi)}(u - \vp_\e) + (\mu_2+\mu_2') e^{-B \psi}|z|^2.
\end{equation*}
At $0$, both sides are 0, so
\begin{equation*}
|D_{2n}D_\gamma U| \leq \left|D_{2n}\left( (\mu_1+\mu_1') e^{-B(F+\psi)}(u - \vp_\e) + (\mu_2+\mu_2') e^{-B \psi}|z|^2 \right)\right|.
\end{equation*}
Combining this with \eqref{derivatives} and Lemma \ref{first order estimate}, we see that
\begin{equation*}
  |D_{2n}D_\gamma U| \leq CBe^{-(B+B_0)(F+\psi)}.
\end{equation*}
Recalling $U=u-h_{\ve}$ and using (\ref{estimates for h ve}), we then have
\begin{equation*}
  |D_{2n}D_{\gamma}u| \leq |D_{2n}D_\gamma U|+|D_{2n}D_{\gamma} h_\e| \leq CBe^{-(B+B_0)(F+\psi)},
\end{equation*}
as desired.

We will show \eqref{tangent-normal claim} by the minimum principle. Recalling that $u\geq\vp_{\ve}$ and $\xi\geq0$, on $\de B_{R}\cap M$, using Lemma \ref{first order estimate},
\[
w \geq e^{BF}\left(\mu_{2}'R^{2}-Ce^{(B-2B_{0})\psi}\right).
\]
On $B_{R}\cap \de M$, since $D_{\gamma}U=0$, then
\[
w \geq \mu_{2}'e^{BF}|z|^{2} \geq 0.
\]
Thus, choosing $\mu_2'$ large enough ensures
\[
w\geq 0 ~\text{on $\partial (B_R\cap M)$}.
\]

Suppose then that $x_0$ is an interior minimum point of $w$ -- if it is on the boundary, then we are done by the above. Additionally, if we have that $D_\gamma U(x_0) = 0$, $e^{F(x_0)} = 0$, or $e^{\psi(x_0)} = 0$, then clearly $w(x_0) \geq 0$ also (recall that $\xi\geq0$ on $B_{R}\cap M$), and hence $w\geq 0$ on all of $B_R\cap M$. Thus we may assume that $x_0$ is an interior minimum such that $D_\gamma U(x_0)\neq 0$, $e^{F(x_0)} > 0$, and $e^{\psi(x_0)} > 0$ -- we will assume that $D_\gamma U(x_{0})<0$ here, as the alternative is basically the same. This ensures that now $w$ is at least $C^2$ at $x_0$, so we may take its Laplacian.

At $x_{0}$, we wish to compute
\begin{equation}\label{tangent-normal equ 1}
\begin{split}
\Delta_{\ti{g}}w = {} & \mu_{1}'\Delta_{\ti{g}}(u-\vp_\e) + \mu_2'\Delta_{\ti{g}}(e^{BF}|z|^2) +\Delta_{\ti{g}}\left(e^{B(F+\psi)}(D_\gamma U)\right)\\
{} &  - \Delta_{\ti{g}}\left(e^{B(F+\psi)}(D_\gamma U)^{2}\right) + \Delta_{\ti{g}} \xi.
\end{split}
\end{equation}
By the same calculations as for \eqref{tangent-normal first term} and \eqref{tangent-normal second term}, the first two terms of (\ref{tangent-normal equ 1}) can be controlled by:
\begin{equation}\label{tangent-normal equ 5}
\begin{split}
\mu_{1}'\Delta_{\ti{g}}(u-\vp_\e) + &\mu_2'\Delta_{\ti{g}}(e^{BF}|z|^2) \\
\leq {} & -\mu_1'(e^F + \e/4)\tr{\ti{\omega}}{\omega}+CB^{2}e^{(B-B_0)F}\tr{\ti{\omega}}{\omega}.
\end{split}
\end{equation}
For the third term of (\ref{tangent-normal equ 1}), by $U=u-h_{\ve}$, \eqref{derivatives}, \eqref{estimates for h ve}, Lemma \ref{first order estimate} and the Cauchy-Schwarz inequality, we have
\begin{equation}\label{tangent-normal equ 2}
\begin{split}
        \Delta_{\ti{g}}&\left(e^{B(F+\psi)}(D_\gamma U)\right) \\[1mm]
   = {} & e^{B(F+\psi)}\Delta_{\ti{g}}(D_\gamma U)+(D_\gamma U)\Delta_{\ti{g}}\left(e^{B(F+\psi)}\right) \\[1mm]
        & +Be^{B(F+\psi)}\ti{g}^{i\ov{j}}\left((F+\psi)_{i}(\de_{\ov{j}}D_{\gamma}U)+(F+\psi)_{\ov{j}}(\de_{i}D_{\gamma}U)\right) \\[1mm]
\leq {} & e^{B(F+\psi)}\Delta_{\ti{g}}(D_\gamma u)-e^{B(F+\psi)}\Delta_{\ti{g}}(D_{\gamma}h_\e)
         +CB^{2}e^{(B-B_0)(F+\psi)}\tr{\ti{\omega}}{\omega} \\[1mm]
        & +e^{B(F+\psi)}\ti{g}^{i\ov{j}}(\de_{i}D_{\gamma}U)(\de_{\ov{j}}D_{\gamma}U)+B^{2}e^{B(F+\psi)}\ti{g}^{i\ov{j}}(F+\psi)_{i}(F+\psi)_{\ov{j}} \\[1mm]
\leq {} & e^{B(F+\psi)}\Delta_{\ti{g}}(D_{\gamma}u)+CB^{2}e^{(B-B_0)(F+\psi)}\tr{\ti{\omega}}{\omega}
          +e^{B(F+\psi)}\ti{g}^{i\ov{j}}(\de_{i}D_{\gamma}U)(\de_{\ov{j}}D_{\gamma}U).
\end{split}
\end{equation}
For the third order term $\Delta_{\ti{g}}(D_\gamma u)$, recalling the definition of $D_{\gamma}$, we have
\begin{equation*}
  D_{\gamma} = \frac{\de}{\de x_{\gamma}}+a\frac{\de}{\de x_{2n}}, ~\text{~where~} a = -\frac{r_{x_{\gamma}}}{r_{x_{2n}}}.
\end{equation*}
It follows then from a direct calculation that
\[
\Delta_{\ti{g}}(D_{\gamma}u)
= \tr{\ti{\omega}}{(D_{\gamma}\ddbar u)} + u_{x_{2n}}\Delta_{\ti{g}}a
+2\mathrm{Re}\left(\ti{g}^{i\ov{j}}a_{i}u_{x_{2n}\ov{j}}\right),
\]
where
\[
u_{x_{2n}\ov{j}} = \frac{\de}{\de\ov{z}_{j}}\left(\frac{\de u}{\de x_{2n}}\right).
\]
In the following argument, we use subscripts to denote partial derivatives for convenience. Since $\frac{\de}{\de x_{2n}}=\sqrt{-1}\left(2\frac{\de}{\de z_{n}}-\frac{\de}{\de x_{2n-1}}\right)$, we have
\[
\begin{split}
2\mathrm{Re}\left(\ti{g}^{i\ov{j}}a_{i}u_{x_{2n}\ov{j}}\right)
= {} & -2\mathrm{Im}\left(\ti{g}^{i\ov{j}}a_{i}(2u_{n\ov{j}}-u_{\ti{\gamma} \ov{j}})\right) \\
= {} & -4\mathrm{Im}\left(\ti{g}^{i\ov{j}}a_{i}(\ti{g}_{n\ov{j}}-\alpha_{n\ov{j}}-\ve g_{n\ov{j}})\right)
+2\mathrm{Im}\left(\ti{g}^{i\ov{j}}a_{i}u_{\ti{\gamma} \ov{j}}\right)
\end{split}
\]
(recall that $\ti{\gamma}  = 2n-1$). Applying $D_{\gamma}$ to (\ref{ma}) gives,
\begin{equation*}
  \tr{\ti{\omega}}{(D_{\gamma}(\alpha + {\ve\omega})+D_{\gamma}\ddbar u)} = \beta (D_{\gamma}U) + \tr{\omega}{(D_{\gamma} \omega)}.
\end{equation*}
Hence,
\begin{equation*}
\begin{split}
\Delta_{\ti{g}}(D_{\gamma}u)
& = {} \beta(D_{\gamma}U) + \tr{\omega}{(D_\gamma\omega)} - \tr{\ti{\omega}}(D_{\gamma}(\alpha + {\ve\omega})) \\
     & +u_{x_{2n}}\Delta_{\ti{g}}a -4\mathrm{Im}\left(\ti{g}^{i\ov{j}}a_{i}(\ti{g}_{n\ov{j}}-\alpha_{n\ov{j}}-\ve g_{n\ov{j}})\right) +2\mathrm{Im}\left(\ti{g}^{i\ov{j}}a_{i}u_{\ti{\gamma} \ov{j}}\right).
\end{split}
\end{equation*}
Combining this with \eqref{estimates for h ve}, Lemma \ref{first order estimate}, $\tr{\ti{\omega}}{\omega}\geq 4n$ (cf. (\ref{lower bound of trace})),  $U=u-h_{\ve}$, and the Cauchy-Schwarz inequality, we see that:
\begin{equation}\label{tangent-normal equ 4}
\begin{split}
\Delta_{\ti{g}}(D_{{\gamma}}u)
\leq {} & \beta D_{\gamma}U + Ce^{-B_0\psi}\tr{\ti{\omega}}{\omega}+2\mathrm{Im}\left(\ti{g}^{i\ov{j}}a_{i}u_{\ti{\gamma}\ov{j}}\right) \\
\leq {} & \beta D_{\gamma}U + Ce^{-B_0\psi}\tr{\ti{\omega}}{\omega}+2\mathrm{Im}\left(\ti{g}^{i\ov{j}}a_{i}(\p_{\ti{\gamma}}\p_{\ov{j}} U)\right) \\
\leq {} & \beta D_{\gamma}U + Ce^{-B_0\psi}\tr{\ti{\omega}}{\omega} + C\sqrt{\tr{\ti{\omega}}{\omega}}\sqrt{\ti{g}^{i\ov{j}}(\p_{\ti{\gamma}}\p_{i} U)( \p_{\ti{\gamma}}\p_{\ov{j}} U)}.
\end{split}
\end{equation}
We also have the corresponding lower bound:
\begin{equation}\label{tangent-normal equ 6}
\Delta_{\ti{g}}(D_{{\gamma}}u) \geq  \beta D_{\gamma}U - Ce^{-B_0\psi}\tr{\ti{\omega}}{\omega}
-C\sqrt{\tr{\ti{\omega}}{\omega}}\sqrt{\ti{g}^{i\ov{j}}(\p_{\ti{\gamma}}\p_{i} U)( \p_{\ti{\gamma}}\p_{\ov{j}} U)},
\end{equation}
Substituting \eqref{tangent-normal equ 4} into \eqref{tangent-normal equ 2} and using the Cauchy-Schwarz inequality, we get the estimate:
\begin{equation}\label{tangent-normal third term}
\begin{split}
\Delta_{\ti{g}}&\left(e^{B(F+\psi)}D_\gamma U\right) \\[1mm]
\leq {} & \beta e^{B(F+\psi)}D_\gamma U+CB^{2}e^{(B-B_0)(F+\psi)}\tr{\ti{\omega}}{\omega} \\[1mm]
& +\frac{1}{4}e^{B(F+\psi)}\ti{g}^{i\ov{j}}(\p_{\ti{\gamma}}\p_{i} U)( \p_{\ti{\gamma}}\p_{\ov{j}} U)
+e^{B(F+\psi)}\ti{g}^{i\ov{j}}(\de_{i}D_{\gamma}U)(\de_{\ov{j}}D_{\gamma}U).
\end{split}
\end{equation}
For the fourth term of (\ref{tangent-normal equ 1}), using $U=u-h_{\ve}$, (\ref{estimates for h ve}), (\ref{tangent-normal equ 4}), (\ref{tangent-normal equ 6}), Lemma \ref{first order estimate} and the Cauchy-Schwarz inequality, at the expense of increasing $B_{0}$, we compute
\begin{equation*}
\begin{split}
-\Delta_{\ti{g}}&\left((D_\gamma U)^{2}\right) \\
= {} & -2(D_{\gamma}U)\Delta_{\ti{g}}(D_\gamma U)-2\ti{g}^{i\ov{j}}(\de_{i}D_{\gamma}U)(\de_{\ov{j}}D_{\gamma}U) \\
= {} & -2(D_{\gamma}U)\Delta_{\ti{g}}(D_\gamma u)+2(D_{\gamma}U)\Delta_{\ti{g}}(D_{\gamma}h_\e)-2\ti{g}^{i\ov{j}}(\de_{i}D_{\gamma}U)(\de_{\ov{j}}D_{\gamma}U) \\
\leq {} & -2\beta(D_{\gamma}U)^{2}+Ce^{-B_0\psi}\tr{\ti{\omega}}{\omega}
+\frac{1}{4}\ti{g}^{i\ov{j}}(\p_{\ti{\gamma}}\p_{i} U)( \p_{\ti{\gamma}}\p_{\ov{j}} U)-2\ti{g}^{i\ov{j}}(\de_{i}D_{\gamma}U)(\de_{\ov{j}}D_{\gamma}U),
\end{split}
\end{equation*}
so that we have
\begin{equation}\label{tangent-normal fourth term}
\begin{split}
-\Delta_{\ti{g}}&\left(e^{B(F+\psi)}(D_\gamma U)^{2}\right) \\[1mm]
= {} & -(D_\gamma U)^{2}\Delta_{\ti{g}}(e^{B(F+\psi)})-e^{B(F+\psi)}\Delta_{\ti{g}}\left((D_\gamma U)^{2}\right) \\
& -2Be^{B(F+\psi)}\ti{g}^{i\ov{j}}\left((F+\psi)_{i}(\de_{\ov{j}}D_{\gamma}U)(D_{\gamma}U)+(F+\psi)_{\ov{j}}(\de_{\ov{i}}D_{\gamma}U)(D_{\gamma}U)\right) \\
\leq {} & -2\beta e^{B(F+\psi)}(D_{\gamma}U)^{2}+CB^2e^{(B-B_0)(F+\psi)}\tr{\ti{\omega}}\omega
+\frac{1}{4}e^{B(F+\psi)}\ti{g}^{i\ov{j}}(\p_{\ti{\gamma}}\p_{i} U)( \p_{\ti{\gamma}}\p_{\ov{j}} U) \\
& -2e^{B(F+\psi)}\ti{g}^{i\ov{j}}(\de_{i}D_{\gamma}U)(\de_{\ov{j}}D_{\gamma}U)
+e^{B(F+\psi)}\ti{g}^{i\ov{j}}(\de_{i}D_{\gamma}U)(\de_{\ov{j}}D_{\gamma}U) \\
&+4B^{2}e^{B(F+\psi)}(D_{\gamma}U)^{2}\ti{g}^{i\ov{j}}(F+\psi)_{i}(F+\psi)_{\ov{j}} \\
\leq {} & CB^{2}e^{(B-B_0)(F+\psi)}\tr{\ti{\omega}}\omega-e^{B(F+\psi)}\ti{g}^{i\ov{j}}(\de_{i}D_{\gamma}U)(\de_{\ov{j}}D_{\gamma}U)
+\frac{1}{4}e^{B(F+\psi)}\ti{g}^{i\ov{j}}(\p_{\ti{\gamma}}\p_{i} U)( \p_{\ti{\gamma}}\p_{\ov{j}} U).
\end{split}
\end{equation}

Now, substituting  (\ref{Delta xi}), (\ref{tangent-normal equ 5}), (\ref{tangent-normal third term}), and (\ref{tangent-normal fourth term}) into (\ref{tangent-normal equ 1}), at $x_{0}$, we obtain
\begin{equation*}
\begin{split}
\Delta_{\ti{g}}w \leq {} &
-\mu_{1}'(e^F + \e/4)\tr{\ti{\omega}}{\omega} + CB^{2}e^{(B-B_0)(F+\psi)}\tr{\ti{\omega}}{\omega}+\beta e^{B(F+\psi)}D_\gamma U \\
& -\frac{1}{2}e^{B(F+\psi)}\ti{g}^{i\ov{j}}(\p_{\ti{\gamma}}\p_{i} U)( \p_{\ti{\gamma}}\p_{\ov{j}} U).
\end{split}
\end{equation*}
It then follows that
\[
\Delta_{\ti{g}}w \leq
-\mu_{1}'(e^F + \e/4)\tr{\ti{\omega}}{\omega} + CB^{2}e^{(B-B_0)(F+\psi)}\tr{\ti{\omega}}{\omega}+\beta e^{B(F+\psi)}D_\gamma U.
\]
Using then the fact that $D_{\gamma}U(x_{0})<0$, we see that
\begin{equation*}
  \Delta_{\ti{g}}w(x_0) \leq -\mu_{1}'e^{F}\tr{\ti{\omega}}\omega+CB^{2}e^{(B-B_0)(F+\psi)}\tr{\ti{\omega}}\omega.
\end{equation*}
Choosing $B$, $\mu_{1}'$ sufficiently large, it then follows from the fact that $e^{F(x_0)} > 0$ that
\begin{equation*}
\Delta_{\ti{g}}w(x_0) < 0.
\end{equation*}
But since $x_{0}$ was assumed to be an interior minimum, we must have:
\[
\Delta_{\ti{g}} w(x_0) \geq 0,
\]
which is a contradiction. Hence, \eqref{tangent-normal claim} follows.

\bigskip
\noindent
{\bf Step 3.} The normal-normal derivatives.

\bigskip

By steps 1 and 2 we have
\begin{equation*}\label{normal-normal equ 1}
  |D_{\gamma}D_{\eta}u(p)| + |D_{\gamma}D_{2n}u(p)| \leq Ce^{-B_0(F+\psi)} ~\text{~for $1\leq\gamma,\eta\leq 2n-1$}.
\end{equation*}
Thus, to bound the normal-normal derivative, it is sufficient to bound $|u_{n\ov{n}}|$.

Expanding out the determinant $\det(\ti{g}_{i\ov{j}})_{1\leq i,j\leq n}$, we see that we already have the bound
\begin{equation}\label{uh}
  |\det(\ti{g}_{i\ov{j}})_{1\leq i,j\leq n}-\ti{g}_{n\ov{n}}\det(\ti{g}_{i\ov{j}})_{1\leq i,j\leq n-1}| \leq Ce^{-B_0(F+\psi)}.
\end{equation}
Recalling \eqref{ma} and $u-h_{\ve}\leq0$, it is clear that
\begin{equation*}
  \det(\ti{g}_{i\ov{j}})_{1\leq i,j\leq n} = e^{\beta(u-h_\e)+n\log(\ve/4)}\det(g_{i\ov{j}})_{1\leq i,j\leq n} \leq C,
\end{equation*}
so that \eqref{uh} implies
\begin{equation}\label{normal-normal equ 2}
  \ti{g}_{n\ov{n}}\det(\ti{g}_{i\ov{j}})_{1\leq i,j\leq n-1} \leq Ce^{-B_0(F+\psi)}.
\end{equation}

Next we show that there is a uniform lower bound for $\det(\ti{g}_{i\ov{j}})_{1\leq i,j\leq n-1}$. Note that the holomorphic tangent bundle to $\partial M$ at $p$, denoted by $T_{\de M}^{h}$, is spanned by $\{\frac{\de}{\de z_{i}}\}_{i=1}^{n-1}$. Then
\begin{equation}\label{normal-normal equ 3}
\begin{split}
\ti{\omega}|_{T_{\de M}^{h}} = {} & (\alpha+\ve\omega+\ddbar u)|_{T_{\de M}^{h}} \\
                             = {} & (\alpha +\ve\omega + \ddbar\vp_\e)|_{T_{\de M}^{h}}+\ddbar(u-\vp_\e)|_{T_{\de M}^{h}} \\
                          \geq {} & e^{F}\omega|_{T_{\de M}^{h}}+\ddbar(u-\vp_\e)|_{T_{\de M}^{h}},
\end{split}
\end{equation}
where we used \eqref{key'} in the last inequality. Since $u-\vp_\e\equiv 0$ on $\de M$, by \cite[Lemma 7.3]{Bou},
\begin{equation*}\label{normal-normal equ 4}
  \ddbar(u-\vp_\e)|_{T_{\de M}^{h}} = \left(\nu\cdot(u-\vp_\e)\right)L_{\de M,\nu},
\end{equation*}
where $\nu$ is an outward pointing normal vector field on $\de M$ and $L_{\de M,\nu}$ is the corresponding Levi-form of $\de M$. Recalling that $\de M$ is weakly pseudoconcave, we have $L_{\de M,\nu}\leq 0$. Since $u-\vp_\e\geq0$ on $M$ and $u-\vp_\e\equiv 0$ on $\de M$, we have
\begin{equation*}
  \nu\cdot(u-\vp_\e) \leq 0,
\end{equation*}
so \eqref{normal-normal equ 3} implies
\begin{equation*}
  \ti{\omega}|_{T_{\de M}^{h}} \geq e^{F}\omega|_{T_{\de M}^{h}}.
\end{equation*}
Taking wedges, we then get
\begin{equation*}
  \det(\ti{g}_{i\ov{j}})_{1\leq i,j\leq n-1} \geq \frac{1}{C}e^{(n-1)F}.
\end{equation*}
Combining this with (\ref{normal-normal equ 2}) and the definition of $\ti{\omega}$ we have
\begin{equation*}
  |u_{n\ov{n}}| = |\ti{g}_{n\ov{n}}-\alpha_{n\ov{n}}- \e g_{n\ov{n}}| \leq Ce^{-B(F+\psi)},
\end{equation*}
at $p$, as desired.
\end{proof}

We can now bound the Laplacian on the interior:

\begin{proposition}\label{Laplacian estimate}
 Assume we are in the situation in Proposition \ref{boundary second order estimate}. Then there exist uniform constants $\beta_0,B$, and $C > 0$ such that:
\[
|\Delta_g u| \leq Ce^{-B\ti{\psi}}\text{ for all }\beta \geq \beta_0,
\]
where here $\widetilde{\psi}$ is as in Proposition \ref{forms}.
\end{proposition}

\begin{proof}
We may assume without loss of generality that $F\leq0$. By the construction of $\ti{\psi}$, we have
\[
\ti{\psi} \leq \psi \leq 0  \ \text{and} \ \alpha+\ddbar\ti{\psi} \geq \frac{\delta}{2}\omega.
\]
Recalling $u-\psi\geq0$ and (\ref{WLOG}), it then follows that
\begin{equation*}\label{Laplacian estimate equ 1}
u-\ti{\psi} \geq 0
\end{equation*}
and
\begin{equation}\label{Laplacian estimate equ 2}
\alpha+\ve\omega+\left(1+\delta/2\right)\ddbar\ti{\psi}
\geq \left(1+\delta/2\right)\frac{\delta}{2}\omega-\frac{\delta}{2}\omega
\geq \frac{\delta^{2}}{4}\omega.
\end{equation}

The trick is to again use:
\begin{equation*}\label{definition of tilde u}
  \ti{u} := u - (1+\delta/2)\ti{\psi}.
\end{equation*}
By (\ref{Laplacian estimate equ 2}), it is clear that
\begin{equation}\label{Laplacian estimate equ 3}
-\Delta_{\ti{g}}\ti{u} \geq -n+\frac{\delta^{2}}{4}\tr{\ti{\omega}}\omega.
\end{equation}
Consider the following quantity:
 \[
 Q = \log \tr{\omega}{\ti\omega} - B\ti{u},
 \]
where $B$ is a constant to be determined later. We will bound $Q$ above using the maximum principle. Let $x_0$ be a maximum point of $Q$. It suffices to prove
\[
(\tr{\omega}{\ti{\omega}})(x_{0}) \leq Ce^{-C\ti{\psi}(x_{0})} \ \ \text{for some $C$},
\]
as then:
\[
Q(x_{0}) \leq \log C-C\ti{\psi}(x_{0})-B\ti{u}(x_{0})
\leq\log C+(B\delta/2-C)\ti{\psi}(x_{0}) \leq C,
\]
as long as $B\geq 2C/\delta$.

Now, if $x_0\in\partial M$, then we are already done by Proposition \ref{boundary second order estimate}, as:
\[
\tr{\omega}{\ti\omega} = \tr{\omega}{(\alpha + \e\omega)} + \Delta_g u
\leq Ce^{-B(\psi+F)} \leq Ce^{-C\ti{\psi}} \ \ \text{on $\de M$}.
\]
Note also that $x_0$ cannot occur on $\mathrm{Sing}(\psi)$. We may then compute at $x_0$, using (\ref{Laplacian estimate equ 3}) and the estimate of \cite{A,Y}:

\begin{equation*}
\begin{split}
0 \geq \Delta_{\ti g} Q(x_0)
\geq {}
& \frac{1}{\tr{\omega}{\ti\omega}} (-C(\tr{\ti\omega}{\omega})(\tr{\omega}{\ti\omega}) - \tr{\omega}{\Ric(\ti\omega)})
-Bn+\frac{B\delta^{2}}{4}\tr{\ti{\omega}}\omega \\
\geq {}
& \left(\frac{B\delta^2}{4} -C\right)\tr{\ti\omega}{\omega} - \frac{\tr{\omega}{(\Ric(\omega) - \beta\ddbar(u - h_\e))}}{\tr{\omega}{\ti\omega}} - Bn \\
\geq {}
& \frac{-\tr{\omega}{(\Ric(\omega) - \beta\ti\omega +\beta(\alpha + \e\omega))} - \beta Ce^{-B_0\psi}}{\tr{\omega}{\ti\omega}} - Bn \\
\geq {}
& \frac{\beta}{2} + \frac{-C\beta - C\beta e^{-B_0\psi}}{\tr{\omega}{\ti\omega}}
\end{split}
\end{equation*}
for $B,\beta$ sufficiently large. Rearranging gives:
\[
\frac{C + C e^{-B_0\psi(x_0)}}{\tr{\omega}{\ti\omega}(x_0)} \geq \frac{1}{2}.
\]
It then follows that
\[
\tr{\omega}{\ti\omega}(x_0)
\leq 2C(1+e^{-B_0\psi(x_0)})\
\leq 4Ce^{-B_0\ti\psi(x_0)}
\]
as required.

Thus, we conclude that:
\[
Q\leq C
\]
for a uniform $C$. It follows that:
\[
\tr{\omega}{\ti{\omega}} \leq Ce^{-B\ti\psi},
\]
and hence:
\[
(\Delta_{g}u) = (\tr{\omega}{\ti{\omega}})-\tr{\omega}(\alpha+\ve\omega) \leq Ce^{-B\ti{\psi}}.
\]

\end{proof}

\begin{proposition}\label{Hess}
 Assume we are in the situation in Proposition \ref{boundary second order estimate}. Then there exist uniform constants $\beta_0, B,$ and $C > 0$ such that
\[
|\nabla^{2}u|_{g} \leq Ce^{-B\ti{\psi}}\text{ for all }\beta > \beta_0,
\]
where $\ti{\psi}$ is as in Proposition \ref{forms}.
\end{proposition}
\begin{proof}
We already have that the Hessian is bounded on the boundary by Proposition \ref{boundary second order estimate}. We may then apply the maximum principle argument in \cite[Lemma 4.3]{CTW} using $\ti{\psi}$ instead of $\psi$, which gives us the estimate everywhere, as desired. Note that, although it is assumed in \cite{CTW} that $\psi$ has analytic singularities, it is easy to see that the proof only needs the weaker assumption of exponential smoothness, in the specific form of \eqref{derivatives}.

For the reader's convenience, we give a brief sketch here. Recalling Lemma \ref{first order estimate} and Propositions \ref{boundary second order estimate} and \ref{Laplacian estimate}, it is clear that
\begin{equation}\label{Hess equ 4}
\sup_{M}(e^{B_{0}\ti{\psi}}|\nabla u|_{g}^{2}+e^{B_{0}\ti{\psi}}|\Delta u|)+\sup_{\de M}(e^{B_{0}\ti{\psi}}|\nabla^{2}u|_{g}) \leq C.
\end{equation}
Without loss of generality, we assume that $\ti{\psi}\leq-1$. We consider the following quantity:
\[
Q = \log\lambda_{1}+\rho(e^{B\ti{\psi}}|\nabla u|_{g}^{2})-A\ti{u},
\]
where $\lambda_{1}$ is the largest eigenvalue of the real Hessian $\nabla^{2}u$, $\ti{u}$ is as in Proposition \ref{Laplacian estimate}, $A$ and $B$ are positive constants to be determined, and the function $\rho$ is given by
\[
\rho(s) = -\frac{1}{2}\log\left(1+\sup_{M}(e^{B\ti{\psi}}|\nabla u|_{g}^{2})-s\right).
\]
Let $x_{0}$ be a maximum point of $Q$. By a similar argument to that in Proposition \ref{Laplacian estimate}, it suffices to prove
\[
\lambda_{1}(x_{0}) \leq e^{-C\ti{\psi}(x_{0})}.
\]
If $x_{0}\in\de M$, then we are done by (\ref{Hess equ 4}). Thus, we assume that $x_{0}$ is an interior point and $Q$ is smooth at $x_{0}$ (otherwise, we just need to apply a perturbation argument as in \cite{CTW}). We compute everything at $x_{0}$. Choose holomorphic normal coordinates such that
\[
g_{i\ov{j}} = \delta_{ij}, \quad
\ti{g}_{i\ov{j}} = \delta_{ij}\ti{g}_{i\ov{i}}, \quad
\ti{g}_{1\ov{1}} \geq \cdots \geq \ti{g}_{n\ov{n}} \quad
\text{at $x_0$}.
\]
Applying $\de_{k}$ to the logarithm of (\ref{ma}), we have
\[
\ti{g}^{i\ov{i}}\de_{k}(\ti{g}_{i\ov{i}}) = \beta(u_{k}-(h_{\ve})_{k}).
\]
It then follows that
\begin{equation*}
\begin{split}
& \Delta_{\ti{g}}(e^{B\ti{\psi}}|\nabla u|_{g}^{2}) \\
\geq {}
& \frac{e^{B\ti{\psi}}}{2}\sum_{k}\ti{g}^{i\ov{i}}\left(|u_{ik}|^{2}+|u_{i\ov{k}}|^{2}\right)
-CB^{2}e^{(B-C)\ti{\psi}}\sum_{i}\ti{g}^{i\ov{i}}-C\beta e^{(B-C)\ti{\psi}} \\
\geq {}
& \frac{e^{B\ti{\psi}}}{2}\sum_{k}\ti{g}^{i\ov{i}}\left(|u_{ik}|^{2}+|u_{i\ov{k}}|^{2}\right)
-\sum_{i}\ti{g}^{i\ov{i}}-\beta,
\end{split}
\end{equation*}
after choosing $B$ sufficiently large such that $CB^{2}e^{(B-C)\ti{\psi}}\leq1$. Since $\rho'\leq\frac{1}{2}$, we obtain
\begin{equation}\label{Hess equ 2}
\begin{split}
\Delta_{\ti{g}}(\rho(e^{B\ti{\psi}}|\nabla u|_{g}^{2}))
& \geq \frac{\rho'e^{B\ti{\psi}}}{2}\sum_{k}\ti{g}^{i\ov{i}}\left(|u_{ik}|^{2}+|u_{i\ov{k}}|^{2}\right) \\
& +\rho''\ti{g}^{i\ov{i}}|\de_{i}(e^{B\ti{\psi}}|\nabla u|_{g}^{2})|^{2}-\frac{1}{2}\sum_{i}\ti{g}^{i\ov{i}}-\frac{\beta}{2}.
\end{split}
\end{equation}

As in \cite[pg. 297]{CTW}, let $V_{\alpha}$ be the unit eigenvector corresponding to $\lambda_{\alpha}$ (the eigenvalues of $\nabla^{2}u$ at $x_{0}$). Extend each $V_{\alpha}$ to a vector field near $x_{0}$ with constant coefficients. Applying $V_{1}V_{1}$ to the logarithm of (\ref{ma}) and using $V_{1}V_{1}(u)=\lambda_{1}$, we have
\[
\begin{split}
\ti{g}^{i\ov{i}}V_{1}V_{1}(\ti{g}_{i\ov{i}}) = {}
& \ti{g}^{p\ov{p}}\ti{g}^{q\ov{q}}|V_{1}(\ti{g}_{p\ov{q}})|^{2}+V_{1}V_{1}(\log\det g)+\beta V_{1}V_{1}(u-h_{\ve}) \\
\geq {}
& \ti{g}^{p\ov{p}}\ti{g}^{q\ov{q}}|V_{1}(\ti{g}_{p\ov{q}})|^{2}-C+\beta(\lambda_{1}-Ce^{-C\ti{\psi}}),
\end{split}
\]
where we used (\ref{estimates for h ve}) and $\ti{\psi}\leq\psi$ in the second inequality. Without loss of generality, we assume that $\lambda_{1}\geq4Ce^{-C\ti{\psi}}+4C$. It then follows that
\[
\ti{g}^{i\ov{i}}V_{1}V_{1}(\ti{g}_{i\ov{i}})
\geq \ti{g}^{p\ov{p}}\ti{g}^{q\ov{q}}|V_{1}(\ti{g}_{p\ov{q}})|^{2}+\frac{\beta}{2},
\]
which implies
\begin{equation}\label{Hess equ 3}
\begin{split}
\Delta_{\ti{g}}(\log\lambda_{1})
\geq {}
2\sum_{\alpha>1}\frac{\ti{g}^{i\ov{i}}|\de_{i}(u_{V_{\alpha}V_{1}})|^{2}}{\lambda_{1}(\lambda_{1}-\lambda_{\alpha})}
+\frac{\ti{g}^{p\ov{p}}\ti{g}^{q\ov{q}}|V_{1}(\ti{g}_{p\ov{q}})|^{2}}{\lambda_{1}}
-\frac{\ti{g}^{i\ov{i}}|\de_{i}(u_{V_{1}V_{1}})|^{2}}{\lambda_{1}^{2}}+\frac{\beta}{2}.
\end{split}
\end{equation}
Combining (\ref{Laplacian estimate equ 3}), (\ref{Hess equ 2}), (\ref{Hess equ 3}) and the rest of arguments of \cite[Lemma 4.3]{CTW}, we obtain $\lambda_{1}(x_{0})\leq Ce^{-C\ti{\psi}(x_{0})}$, as required.

\end{proof}

We may now finish as follows. By \cite[Proposition 2.3]{Ber}, we have:
\[
 u_{\e,\beta} \xrightarrow{C^0} V_\e
\]
where:
\[
  V_\e := \sup\{v\in\PSH(M,\alpha+\e\omega)\ |\ v\leq h_{\ve}\}.
\]
As mentioned earlier, the $V_{\ve}$ decrease pointwise to $V$ as $\ve$ decreases to $0$. Using (\ref{zero order estimate}), Lemma \ref{first order estimate} and Proposition \ref{Hess}, we establish a uniform $C^{1,1}$ estimate for $u$ on compact subsets away from $\mathrm{Sing}(\psi)$, which implies $V\in C^{1,1}_{\textrm{loc}}(M\setminus \mathrm{Sing}(\psi))$, as required.
\end{proof}

\section{Geodesics between Singular K\"ahler Metrics}\label{geometry}

We now show that our results apply in the setting of regularity of geodesics between singular K\"ahler metrics.

\begin{proof}[Proof of Theorem \ref{cor}]
Let $(X_0,\omega_0)$ be a compact K\"ahler variety, without boundary, and:
\[
  \mu: (X,\omega) \rightarrow (X_0,\omega_0)
\]
a smooth resolution of the singularities of $X_0$ with simple normal crossings, which exists thanks to Hironaka's theorem \cite{Hir}. Let $\mu^{-1}(X_{0, \mathrm{Sing}}) = E = \cup_{k=1}^m E_k$ be the exceptional divisor with smooth irreducible components $E_k$. Let $\alpha_0 := \mu^*\omega_0\geq 0$, which will be a smooth semi-positive form. It is well-known that $[\alpha_0]$ is a big and nef class, and that $E_{nK}(\alpha_0) = \mathrm{Supp}(E)$. Consider smooth Hermitian metrics $h_k$ on $\mathcal O(E_k)$ and defining sections $s_k$ for each $E_k$.

Elementary results in several complex variables will now show:
\begin{equation}\label{bad}
\alpha_0^n \geq b\left(\prod_{k=1}^m |s_k|_{h_k}^{a_k}\right) \omega^n
\end{equation}
for fixed constants $a_k > 0$ and a $b > 0$ depending on $\omega_0$. To see this, we work locally -- cover $X_0$ by open charts $U_i$ such that for each $i$ there exists an embedding:
\[
 \iota_i : U_i \hookrightarrow \Omega_i \subset \C^N,
\]
with $N$ uniformly large, such that $\omega_0$ extends to a smooth K\"ahler form (which we will also call $\omega_0$) on the open set $\Omega_i$. Relabeling $\mu$ to be $\iota_i\circ \mu$, we have that $\mu^*\omega_0$ is unchanged (as the pullback of a composition is the composition of the pullbacks), so we may work with a holomorphic map between smooth spaces. Fix coordinates $z$ on $\mu^{-1}(\Omega_i)$ and $x$ on $\Omega_i$, and define the Jacobian of $\mu$ to be the $n\times N$ matrix:
\[
 \mathrm{Jac}(\mu) := \left(\frac{\partial \mu^k}{\partial z^j}\right)_{\substack{1\leq j\leq n \\ 1\leq k\leq N}},
\]
where $\mu^k$ is the $k^\text{th}$ coordinate function of $\mu$ on $\Omega_i$.

Putting $e^j := \sqrt{-1}dz^j\wedge d\ov{z}^j$, one can then compute that:
\[
 \alpha_0^n(z) = \det(\mathrm{Jac}(\mu)\cdot \omega_0(\mu(z)) \cdot\ov{\mathrm{Jac}(\mu)}^T) n!e^1\wedge\ldots\wedge e^n
\]
where we are expressing $\omega_0(x)$ as an $N\times N$ matrix in the $x$-coordinates. Letting $c > 0$ be a constant such that $\omega_0 \geq c\omega_{\textrm{Eucl}}$ on all charts $\Omega_i$ (which we can do, after possibly shrinking them slightly, as there are only finitely many), we then have:
\[
 \alpha_0^n(z) \geq c^n \det(\mathrm{Jac}(\mu)\cdot\ov{\mathrm{Jac}(\mu)}^T) n!e^1\wedge\ldots\wedge e^n.
\]
By the Cauchy-Binet formula, we have:
\[
\det\left(\mathrm{Jac}(\mu)\cdot\ov{\mathrm{Jac}(\mu)}^T\right) = \sum_{\substack{n\times n\text{ minors}\\ J_k\text{ of }\mathrm{Jac}(\mu)}} |\det(J_k)|^2
\]
so that:
\[
 \alpha_0^n \geq c^n \left(\sum_{\substack{n\times n\text{ minors}\\ J_k\text{ of }\mathrm{Jac}(\mu)}} |\det(J_k)|^2\right)n!e^1\wedge\ldots\wedge e^n.
\]
Each determinant in the sum is a holomorphic function, and furthermore, we know that their common zero locus is $E_{nK}(\alpha_0) = E$, as $\mu$ is a local biholomorphism if and only if $\mathrm{Jac}(\mu)$ has full-rank, which is only true when at least one of the determinants is non-zero. Thus, by the Weierstrass preparation theorem and the fact that $E$ has simple normal crossings, we know that we can express each determinant (locally) as a product of the $s_i$ to some powers, as well as some other local holomorphic functions that do not vanish along all of $E$ -- up to estimating these, the smooth Hermitian metrics, and $\omega^n$, we then see the claim \eqref{bad}.

We may now use the discussion immediately following Proposition \ref{forms} to see that:
\[
\alpha_0 \geq ce^F \omega,
\]
where:
\[
F := \log\left(b\prod_{k=1}^m |s_k|_{h_k}^{a_k}\right),
\]
and $c$ depends on an upper bound for $\alpha_0$ (which always exists as it is a smooth form). Up to shrinking $b$, we may arrange that $F\leq 0$, and note that $F$ has analytic singularities only along $E_{nK}(\alpha_0)$. For a very large constant $C$ then, we have that $F\in\PSH(X, C\omega)$, by the Poincar\'{e}-Lelong formula:
\[
 \ddbar F = \sum_{k=1}^m a_k ([E_k] - R_k) \geq -C\omega,
\]
so we can apply Proposition \ref{forms} to get the key condition \eqref{key}. Note that the resulting $\ti{\psi}$ actually has analytic singularities only along $E_{nK}(\alpha_0)$, so our estimates will be optimal.\\

To now apply this to the geodesic, we will need to translate this onto the product space $X\times A$, where $A$ is the annulus:
\[
 A := \{\tau\in\C\ |\  e^{-1} < |\tau| < 1\}.
\]
Let $\pi$ be the projection onto $X$ and $p$ the projection onto $A$, and define $\alpha := \pi^*\alpha_0$. Throughout, we will use:
\[
t := -\log|\tau|.
\]
Consider two K\"ahler metrics $\omega_1$ and $\omega_2$ on $X_0$ such that $\alpha_1 := \mu^*\omega_1$ and $\alpha_2 :=\mu^*\omega_2$ are cohomologous to $\alpha_0$. Fix a K\"ahler form $\omega$ on $X$ such that
\[
\alpha_k\leq \omega\text{ for }k = 0,1,2,
\]
constants $c, B >0$, and an exponentially smooth, strictly $\alpha_0$-psh $\ti{\psi}$ as above such that:
\[
\alpha_k \geq ce^{B\ti{\psi}}\omega\text{ for }k = 1,2.
\]
There then exist two smooth functions $\vp_1,\vp_2$ such that:
\[
 \alpha_k = \alpha_0 + \ddbar \vp_k,\ \ k= 1,2.
\]
The geodesic between $\alpha_1$ and $\alpha_2$ is then defined to be the envelope:
\[
 V := \sup \{ v\in\PSH(X\times A, \alpha)\ |\ v^{*}|_{\{t = 0\}}\leq \pi^*\vp_1,\ v^{*}|_{\{t = 1\}}\leq \pi^*\vp_2\}.
\]
Fix a large constant $C$ such that $\vp_2 - (C-1) \leq \vp_1 \leq \vp_2 + (C-1)$. Let $f$ be the solution to the Dirichlet problem on $A$:
\[
 \begin{cases}
  \ddbar f = \omega_{\textrm{Eucl}},\\
  f|_{\partial A} = 0.
 \end{cases}
\]
We then define $\vp$ to be the following subsolution:
\[
 \vp(x, \tau) := \widetilde{\max}\{\pi^*\vp_1(x) - Ct , \pi^*\vp_2(x) - C(1-t)\} + p^*f(\tau),
\]
where $\widetilde{\max}$ is a regularized maximum function with error $1/2$. Observe that, for both $k= 1, 2$, on $X\times A$ we have:
\[
 \alpha + \ddbar (\pi^*\vp_k \pm Ct + p^*f) \geq c e^{B\pi^*\ti{\psi}}\pi^*\omega + \sqrt{-1}\partial_\tau\overline{\partial}_\tau(\mp C\log|\tau| + p^*f)
\]
\[
 \geq c e^{B\pi^*\ti{\psi}}  (\pi^*\omega + p^*\omega_{\textrm{Eucl}}).
\]
Thus, by an elementary property of $\widetilde{\max}$ \cite[Lemma 5.18]{Dem}, we have that:
\[
 \alpha + \ddbar \vp \geq ce^{B\pi^*\ti{\psi}} (\pi^*\omega + p^*\omega_{\textrm{Eucl}}),
\]
also, which is the key condition \eqref{key} that we require (as we just take $\vp_\e = \vp$ for all $\e > 0$).
\end{proof}

We now show regularity of geodesics between K\"ahler currents with analytic singularities. We will actually show a more general statement, applicable to potentials which are exponentially smooth and satisfy condition \eqref{key}:

\begin{corollary}\label{cor2}
(Theorem \ref{fake}) Let $(X,\omega)$ be a smooth K\"ahler manifold without boundary, $[\alpha]$ a big and nef class, and $\psi\in\PSH(X,\alpha)$ strictly psh and exponentially smooth. Let $\vp_1,\vp_2\in\PSH(X,\alpha)$ be exponentially smooth functions with the same singularity type such that $\psi\leq \vp_1,\vp_2$ and $\vp_1$ and $\vp_2$ satisfy condition \eqref{key}. Then the geodesic connecting $\vp_1$ and $\vp_2$ is in $C^{1,1}_{\text{\textrm{loc}}}((X\setminus \mathrm{Sing}(\psi))\times A)$, where $A\subset\C$ is the annulus.
\end{corollary}

\begin{proof}
 The idea is to construct a good sequence of K\"ahler potentials for $\{\alpha + \e\omega\}$, and then take regularized maximums with $\vp_1$ and $\vp_2$, which will preserve the estimates we need. Specifically, recall that $(X,\omega)$ is a K\"ahler manifold without boundary, and $[\alpha]$ is a nef and big class on $X$. Our sequence of potentials will be the (smooth) solutions to the following Monge-Amp\`ere equations:
 \begin{equation}\label{this}
  \begin{cases}
   (\alpha + (\e/2)\omega + \ddbar v_\e)^n = e^{\beta_0 v_\e}\omega^n,\\
   \alpha + (\e/2)\omega + \ddbar v_\e > 0,
  \end{cases}
 \end{equation}
 where $\beta_0 > 0$ is a fixed number such that we have the estimates:
 \begin{equation}\label{estimates of v ve}
  |\nabla v_\e|_g^2 + |\nabla^2 v_\e|_g^2 \leq Ce^{-B\psi}\text{ for all }\e > 0.
 \end{equation}
 We can see \eqref{estimates of v ve} by establishing $C^0$ bounds for the $v_\e$ -- it is immediate from the comparison principle \cite[Remark 2.4]{BEGZ} that the $v_\e$ are decreasing as $\e\rightarrow 0$; if $\e_1 > \e_2$, then:
\[
\int_{\{v_{\e_1} < v_{\e_2}\}} e^{\beta_0 v_{\e_2}}\omega^n + 2^{-n}(\e_1 - \e_2)^n\omega^n \leq \int_{\{v_{\e_1} < v_{\e_2}\}} (\alpha + (\e_1/2)\omega + \ddbar v_{\e_2})^n
\]
\[
\leq \int_{\{v_{\e_1} < v_{\e_2}\}} (\alpha + (\e_1/2)\omega + \ddbar v_{\e_1})^n = \int_{\{v_{\e_1} < v_{\e_2}\}} e^{\beta_0 v_{\e_1}}\omega^n \leq \int_{\{v_{\e_1} < v_{\e_2}\}} e^{\beta_0 v_{\e_2}}\omega^n,
\]
which is a contradiction unless $\omega^n(\{v_{\e_1} < v_{\e_2}\}) = 0$, in which case $v_{\e_2} \leq v_{\e_1}$ everywhere, by continuity. In particular, the $v_\e$ are uniformly bounded above by $v_1$. Further, the same argument shows that the $v_\e$ are bounded below by $v_0$ solving:
\begin{equation*}
  \langle(\alpha + \ddbar v_0)^n\rangle = e^{\beta_0 v_0}\omega^n.
\end{equation*}
By \cite[Theorem 6.1]{BEGZ}, $v_0$ has minimal singularities, so there exists a large constant $C$ such that $\psi - C \leq v_0 \leq v_\e$ for all $\e > 0$. The proofs of Lemma \ref{first order estimate} and Proposition \ref{Hess} now apply directly (One can also probably use easier proofs to obtain \eqref{estimates of v ve}, but it follows immediately from what we have done already -- see also \cite[Section 4]{CTW}).

It will be convenient to renormalize the $v_\e$ to be negative by replacing them with $v_\e - \sup_X v_1$, as again $v_\e \leq v_1$. These new $v_\e$ now solve the slightly modified equation:
 \begin{equation}\label{this'}
  \begin{cases}
   (\alpha + (\e/2)\omega + \ddbar v_\e)^n = c_0 e^{\beta_0 v_\e}\omega^n,\\
   \alpha + (\e/2)\omega + \ddbar v_\e > 0,
  \end{cases}
 \end{equation}
for a fixed $c_0 > 0$.

We now pull-back everything to the product space $X\times A$ as in the proof of Theorem \ref{cor} -- define $A$, $t$, $\pi$, $p$ and $f$ as in that proof, and again fix a constant $C$ such that:
\[
  \vp_2 - (C-1)\leq \vp_1 \leq \vp_2 + (C-1),
\]
where we used that $\vp_{1}$ and $\vp_{2}$ have the same singularity type. The geodesic connecting $\vp_{1}$ and $\vp_{2}$ is defined to be
\[
 V := \sup \{ v\in\PSH(X\times A, \pi^{*}\alpha)\ |\ v^{*}|_{\{t = 0\}}\leq \pi^*\vp_1,\ v^{*}|_{\{t = 1\}}\leq \pi^*\vp_2\}.
\]
Note that, if $D_{1}, D_{2}$ are any fixed constants, we have that:
\[
\begin{split}
V +&(D_{2}-D_{1})t+D_{1} = \\
&\sup \{ v\in\PSH(X\times A, \pi^{*}\alpha)\ |\ v^{*}|_{\{t = 0\}}\leq \pi^*\vp_1+D_{1},\ v^{*}|_{\{t = 1\}}\leq \pi^*\vp_2+D_{2}\}.
\end{split}
\]
As $t$ is a smooth function on $X\times A$, we can thus assume without loss of generality that $\sup_X \vp_1 = \sup_X \vp_2 = 0$. After possibly replacing $\psi$ with $\psi - D$ for some large constant $D$, we can also still assume that:
\[
\psi \leq \vp_1, \vp_2.
\]

Now, as in the proof of Theorem \ref{cor}, we define:
\[
  \vp(x,\tau) := \ti{\max}\{\pi^*\vp_1(x) - Ct, \pi^*\vp_2(x) - C(1-t)\} + p^* f (\tau).
\]
To apply Theorem \ref{theorem}, we define the smooth approximation
\[
  \vp_\e(x,\tau) := \ti{\max}\{\pi^*\vp_1(x) - Ct, \pi^*\vp_2(x) - C(1-t), \pi^*v_\e(x) - C_\e\} + p^* f (\tau)
\]
 where $C_\e := -\log(\e/2) + C + 2$. Clearly, $\vp_{\ve}$ decreases pointwise to $\vp$ as $\e$ decreases to $0$. We claim that
\begin{equation}\label{claim 1}
 \pi^*(\alpha + \e\omega) + \ddbar \vp_\e \geq e^{\pi^*\psi}(\pi^*\omega + p^*\omega_{\textrm{Eucl}})
\end{equation}
and
\begin{equation}\label{claim 2}
 |\nabla\vp_{\ve}|_{g}+|\nabla^{2}\vp_{\ve}|_{g} \leq C^{-B\psi}.
\end{equation}
Given these, Corollary \ref{cor2} will follow from Theorem \ref{theorem}.

Let us prove \eqref{claim 1} first. Let $(x_0, \tau_0)\in X\times A$ and $t_0 := -\log |\tau_0|$. If we have
\[
\pi^*\vp_1(x_{0}) - Ct_{0} \leq \pi^*v_\e(x_{0}) - C_\e + 1,
\]
then using that $\psi\leq\vp_{1}$ and $v_\e\leq 0$, it follows that
\[
e^{\pi^{*}\psi} \leq \e/2
\]
near $(x_0, \tau_0)$. Using \eqref{this'}, we see that
\begin{equation*}
\begin{split}
 \pi^*(\alpha + \e\omega) + \ddbar \vp_\e \geq \frac{\ve}{2}\pi^{*}\omega + p^*\omega_{\textrm{Eucl}}
 \geq e^{\pi^*\psi}(\pi^*\omega + p^*\omega_{\textrm{Eucl}}),
\end{split}
\end{equation*}
near $(x_0,\tau_0)$, which implies (\ref{claim 1}) there. If
\[\pi^*\vp_1(x_{0}) - Ct_{0} > \pi^*v_\e(x_{0}) - C_\e+1
\]
on the other hand, then it follows from the definition of $\ti\max$ that:
\[
\vp_\e = \vp
\]
near $(x_0,\tau_0)$, and so:
\[
\pi^*(\alpha + \e\omega) + \ddbar\vp_\e \geq e^{\pi^*\psi}(\pi^*\omega + p^*\omega_{\textrm{Eucl}}),
\]
by our assumptions on $\vp_1$ and $\vp_2$.

We now check (\ref{claim 2}). For ease of notation, write:
\[
b_1 := \pi^*\vp_1 - Ct,\ b_2 := \pi^*\vp_2 - C(1-t),\ b_\e := \pi^{*}v_\e - C_\e.
\]
Recall then the definition of $\ti{\max}$:
 \[
  \ti{\max}(a,b,c) := \int_{\R^3} \max\{y_1,y_2,y_3\}\theta(y_1-a)\theta(y_2-b)\theta(y_3-c)\,dy_1\,dy_2\,dy_3,
 \]
 where here $\theta$, $0\leq \theta\leq 1$ is a cutoff function on $\R$, $\theta \equiv 1$ near $0$, with support in $[-1/2,1/2]$. Thus,
\[
\vp_{\ve} =  \ti{\max}(b_{1},b_{2},b_{\ve})+p^{*}f.
\]
Differentiating $\vp_\e$ once then gives:
\[
\begin{split}
\nabla\vp_{\ve}
= {} & -\nabla b_1\int_{\R^3} \max\{y_1,y_2,y_3\}\theta'(y_1-b_1)\theta(y_2-b_2)\theta(y_3-b_\ve)\,dy_1\,dy_2\,dy_3 \\
& -\nabla b_2\int_{\R^3} \max\{y_1,y_2,y_3\}\theta(y_1-b_1)\theta'(y_2-b_2)\theta(y_3-b_\ve)\,dy_1\,dy_2\,dy_3 \\
& -\nabla b_\ve\int_{\R^3} \max\{y_1,y_2,y_3\}\theta(y_1-b_1)\theta(y_2-b_2)\theta'(y_3-b_\ve)\,dy_1\,dy_2\,dy_3  + \nabla(p^{*}f) \\
=: {} & T_1+T_2+T_3 + O(1),
\end{split}
\]
as $f$ is fixed and smooth. For the first term $T_1$, after changing variables, we have
\[
\begin{split}
|T_1|_g \leq {} & |\nabla b_1|_g\left|\int_{\R^3} \max\{b_1+y_1,b_2+y_2,b_\ve+y_3\}\theta'(y_1)\theta(y_2)\theta(y_3)\,dy_1\,dy_2\,dy_3\right| \\
\leq {} & |\nabla b_1|_g \int_{[-1/2,1/2]^3}|\max\{b_1+y_1,b_2+y_2,b_\ve+y_3\}|\,dy_1\,dy_2\,dy_3,
\end{split}
\]
where we used that $\theta$ has support in $[-1/2,1/2]$. Since $b_1, b_2, b_\e \leq 0$ and $y_1,y_2,y_3\in[-1/2,1/2]$, we have
\[
\begin{split}
0 \geq {} & \max\{b_1+y_1,b_2+y_2,b_\ve+y_3\}-1/2 \\
= {} & -\min\{|b_1+y_1-1/2|,|b_2+y_2-1/2|,|b_\ve+y_3-1/2|\} \\
\geq {} & -\min\{|b_1|+1,|b_2|+1,|b_\ve|+1\} \\
= {} & -\min\{|b_1|,|b_2|,|b_\ve|\}-1,
\end{split}
\]
which implies
\[
|\max\{b_1+y_1,b_2+y_2,b_\ve+y_3\}|
\leq \min\{|b_1|,|b_2|,|b_\ve|\}+1.
\]
Thus,
\[
|T_1|_g \leq C\left(\min\{|b_{1}|,|b_{2}|,|b_{\ve}|\}+1\right)|\nabla b_{1}|_g.
\]
The other terms $T_2$ and $T_3$ can be handled similarly, so that
\[
|\nabla\vp_{\ve}|_{g} \leq C\left(\min\{|b_{1}|,|b_{2}|,|b_{\ve}|\}+1\right)\left(|\nabla b_{1}|+|\nabla b_{2}|+|\nabla b_{\ve}|+1\right).
\]
Using \eqref{estimates of v ve} and the fact that $\vp_{1}$ and $\vp_{2}$ are both exponentially smooth, we then see that
\[
|\nabla\vp_{\ve}|_{g} \leq C(-\pi^{*}\vp_1+C)e^{-B\pi^{*}\psi}\leq C(-\pi^*\psi)e^{-B\pi^*\psi} \leq Ce^{-B\pi^{*}\psi}.
\]
A similar argument shows that $|\nabla^{2}\vp_{\ve}|_{g}\leq Ce^{-B\pi^{*}\psi}$, establishing \eqref{claim 2}.
\end{proof}

We now briefly discuss the case of geodesic rays originating at singular K\"ahler metrics. Recall that the main result of \cite{Mc} can be summarized as follows (see that paper for a more specific statement):

\begin{theorem}\label{above}
 Suppose that $M$ is a compact complex manifold with boundary. Let $\xi$ be a function with analytic singularities on $M$, such that $\xi$ is singular on a divisor $E \subset M$ with $E\cap\partial M = \emptyset$. Let $R_h$ be a smooth form cohomologous to $[E]$, and suppose that $\alpha$ is a closed, smooth, real $(1,1)$-form on $M$ such that $\alpha - R_h$ is $\psi$-big and nef, with $\mathrm{Sing}(\psi)\cap\partial M = \emptyset$. Finally, let $\vp\in \PSH(M,\alpha)$, $\vp \leq \xi$, be smooth near the boundary of $M$ and sufficiently regular. Then the envelope:
 \[
  \sup\{ v\in \PSH(M,\alpha)\ |\ v^{*}|_{\partial M} \leq \vp|_{\partial M},\ v \leq \xi + O(1)\}
 \]
 is in $C^{1,1}_{\textrm{loc}}(M\setminus\mathrm{Sing}(\xi + \psi))$ if the boundary of $M$ is weakly pseudoconcave and $\alpha + \ddbar \vp \geq \delta\omega$ on some neighborhood of $\partial M$.
\end{theorem}

In \cite{Mc}, the above theorem was used to prove $C^{1,1}$ regularity of certain geodesic rays originating at a K\"ahler metric by taking $M = X\times \mathbb{D}$. Here, we simply remark that the results in Section \ref{bignef} can be combined with the method in \cite{Mc} to improve Theorem \ref{above}, and this can then be used to prove regularity of certain geodesic rays on a singular K\"ahler variety in the exact same way:

\begin{theorem}
Theorem \ref{above} is still valid if we allow $\mathrm{Sing}(\psi)$ to intersect $\partial M$ and we weaken the assumptions that $\vp$ be smooth and strictly $\alpha$-psh near the boundary to just assuming the existence of a family of $\vp_\e$ satisfying conditions (a) and (b) in Theorem \ref{theorem}.
\end{theorem}

\appendix

\section{Estimates for $\Delta_\omega$}

\begin{lemma}\label{applem}
Let $\vp_\e$ be as in Theorem \ref{theorem}, and let $h_\e$ be the solutions to:
\[
 \begin{cases}
  \Delta_{2\omega} h_{\ve} = -n, \\
  h_{\ve}|_{\partial M} = \vp_{\ve}|_{\partial M}.
 \end{cases}
\]
for all $\e > 0$. Then there exist positive constants $B, C$ such that:
\begin{equation*}
|\nabla h_\e|_{g} + |\nabla^2 h_\e|_{g} \leq Ce^{-B\psi}.
\end{equation*}
\end{lemma}

\begin{proof}
First, without loss of generality we may scale $\psi$ such that:
\[
 \sup_M\psi = -1.
\]
Now note that by the maximum principle, we have that the $h_\e$ are decreasing as $\e\rightarrow 0$, and that (cf. (\ref{zero order estimate})):
\begin{equation}\label{appC0}
\psi \leq \vp_\e \leq h_\e \leq h_1.
\end{equation}

Let $b$ be the solution to:
\[
 \begin{cases}
  \Delta_{\omega} b = -1,\\
  b|_{\partial M} = 0,
 \end{cases}
\]
and define:
\[
 \ti{h}_\e := h_\e - \vp_\e \geq 0.
\]
We claim that there exists a constant $B > 0$ such that:
\[
 h_\e \leq \vp_\e + e^{-B\psi}b  \ \ \text{on $M\setminus \textrm{Sing}(\psi)$}.
\]
Combining this with the lower bound in \eqref{appC0}, it follows that $|\nabla h_{\ve}|_{g}\leq Ce^{-B\psi}$ at the boundary. To see the claim, using (\ref{derivatives}), we compute:
\begin{equation*}
\begin{split}
\Delta_{\omega}(e^{B\psi}(\ti{h}_{\ve}+\ti{h}_{\ve}^{2}) - b)
& \geq  e^{B\psi}(1+2\ti{h}_{\ve})\Delta_{\omega}\ti{h}_{\ve}+2e^{B\psi}|\nabla\ti{h}_{\ve}|_{g}^{2} \\
& -BCe^{(B-C)\psi}|\nabla\ti{h}_{\ve}|_{g}-CB^{2}e^{(B-C)\psi}(\ti{h}_{\ve}+\ti{h}_{\ve}^{2}) + 1.
\end{split}
\end{equation*}
Since $\Delta_{2\omega}h_{\ve}=-n$, it is clear that $\Delta_{\omega}h_{\ve}=-2n$. Combining this with (\ref{appC0}) and $|\Delta_{\omega}\vp_{\ve}|\leq Ce^{-C\psi}$, we see that
\begin{equation*}
  |\ti{h}_{\ve}| + |\Delta_{\omega}\ti{h}_{\ve}| \leq Ce^{-C\psi}.
\end{equation*}
By the Cauchy-Schwarz inequality, we have
\begin{equation*}
\begin{split}
& \Delta_{\omega}(e^{B\psi}(\ti{h}_{\ve}+\ti{h}_{\ve}^{2}) - b) \\
\geq {} & 2e^{B\psi}|\nabla\ti{h}_{\ve}|_{g}^{2}-e^{B\psi}(|\nabla\ti{h}_{\ve}|_{g}^{2}+CB^{2}e^{-2C\psi})-CB^{2}e^{(B-C)\psi} + 1 \\
\geq {} & -CB^{2}e^{(B-C)\psi} + 1.
\end{split}
\end{equation*}
Now since $\sup_{M}\psi = -1$, we may choose $B$ sufficiently large such that
\begin{equation*}
  \Delta_{\omega}(e^{B\psi}(\ti{h}_{\ve}+\ti{h}_{\ve}^{2})-b) \geq -CB^{2}e^{(B-C)\psi}+1 > 0.
\end{equation*}
It follows then from the maximum principle that
\begin{equation*}
  e^{B\psi}\ti{h}_{\ve} \leq b,
\end{equation*}
as claimed.

To bound the gradient on the interior, we consider the quantity:
\begin{equation*}
  Q = e^{2B\psi}|\nabla h_{\ve}|_{g}^{2} + e^{B\psi}h_{\ve}^{2},
\end{equation*}
where $B$ is a constant to be determined. Suppose that $Q$ achieves a maximum at $x_{0}$. If $x_{0}\in\de M$, then we are already done. Otherwise, we choose holomorphic normal coordinates at $x_{0}$ for $g$ (the Riemannian metric corresponding to $\omega$), so that:
\begin{equation*}
  \sum_{i}(h_{\ve})_{ki\ov{i}} = \sum_{i}(h_{\ve})_{i\ov{i}k} = (\Delta_{\omega}h_{\ve})_{k} = 0.
\end{equation*}
This implies
\begin{equation}\label{part ii equ 4}
\begin{split}
        & \Delta_{\omega}(e^{2B\psi}|\nabla h_{\ve}|_{g}^{2}) \\
\geq {} & 2e^{2B\psi}\textrm{Re}\left(\sum_{i,k}(h_{\ve})_{ki\ov{i}}(h_{\ve})_{\ov{k}}\right)+e^{2B\psi}|\nabla^{2}h_{\ve}|_{g}^{2} \\[1mm]
        & -4Be^{2B\psi}|\nabla\psi|_{g}|\nabla h_{\ve}|_{g}|\nabla^{2}h_{\ve}|_{g}-CB^{2}e^{(2B-C)\psi}|\nabla h_{\ve}|_{g}^{2}\\[2mm]
\geq {} & \frac{1}{2}e^{2B\psi}|\nabla^{2}h_{\ve}|_{g}^{2}-CB^{2}e^{(2B-C)\psi}|\nabla h_{\ve}|_{g}^{2}.
\end{split}
\end{equation}
By the maximum principle, at $x_{0}$ we have
\begin{equation*}
\begin{split}
0 \geq {} & \Delta_{\omega}Q = \Delta_{\omega}( e^{2B\psi}|\nabla h_{\ve}|_{g}^{2})+\Delta_{\omega}(e^{B\psi}h_{\ve}^{2}) \\
  \geq {} & -CB^{2}e^{(2B-C)\psi}|\nabla h_{\ve}|_{g}^{2}+e^{B\psi}|\nabla h_{\ve}|_{g}^{2}-CB^{2}e^{(B-C)\psi}.
\end{split}
\end{equation*}
We can now choose $B$ sufficiently large such that
\begin{equation*}
CB^{2}e^{(2B-C)\psi} \leq \frac{1}{2}e^{B\psi},
\end{equation*}
implying
\begin{equation*}
  |\nabla h_{\ve}|_{g}^{2} \leq Ce^{-C\psi} \ \text{at $x_{0}$}.
\end{equation*}
Using \eqref{appC0} then shows that $Q(x_{0})\leq C$, as desired.

Finally, we bound the Hessian of the $h_\e$. We establish the boundary estimate first. The tangent-tangent derivative estimate is obvious. The tangent-normal derivatives can be bounded in a manner analogous to the proof of Proposition \ref{boundary second order estimate}; very briefly, one considers the quantities:
\begin{equation*}
  \xi = b+e^{B\psi}\left(\mu|z|^{2}-e^{B\psi}\sum_{\alpha=1}^{2n-1}(\partial_{\alpha}\ti{h}_{\ve})^{2}\right),
\end{equation*}
\begin{equation*}
  w = b + e^{B\psi}(\mu|z|^{2}-e^{B\psi}|D_{\gamma}\ti{h}_{\ve}|-e^{B\psi}|D_{\gamma}\ti{h}_{\ve}|^{2})+\xi,
\end{equation*}
on a ball $B_R(p)$ of fixed radius $R$, with $p\in\partial M$, $1\leq\gamma\leq 2n-1$, and $\mu$ and $B$ constants. Choosing $\mu$ sufficiently large,
\begin{equation*}
  \xi \geq 0 ~\text{on }\partial(B_R\cap M).
\end{equation*}
On $(B_{R}\cap M)\setminus\mathrm{Sing}(\psi)$, we compute
\begin{equation*}
\Delta_{\omega}\xi
\leq  -1-\frac{e^{2B\psi}}{C}\sum_{\alpha=1}^{2n-1}\sum_{\beta=1}^{2n}|\partial_{\alpha}\partial_{\beta}\ti{h}_{\ve}|^{2}
+CBe^{(2B-C)\psi}\sum_{\alpha,\beta=1}^{2n}|\partial_{\alpha}\partial_{\beta}\ti{h}_{\ve}|-CB^{2}e^{(B-C)\psi}.
\end{equation*}
Thanks to $\Delta_{\omega}\ti{h}_{\ve}=-2n-\Delta_{\omega}\vp_{\ve}$,
\begin{equation*}
\sum_{\alpha,\beta=1}^{2n}|\partial_{\alpha}\partial_{\beta}\ti{h}_{\ve}|^{2}
\leq C\sum_{\alpha=1}^{2n-1}\sum_{\beta=1}^{2n}|\partial_{\alpha}\partial_{\beta}\ti{h}_{\ve}|^{2}+Ce^{-C\psi}.
\end{equation*}
Hence, at the expense of increasing $B$,
\begin{equation}\label{appC2}
\Delta_{\omega}\xi
\leq  -\frac{1}{2}-\frac{e^{2B\psi}}{C}\sum_{\alpha=1}^{2n-1}\sum_{\beta=1}^{2n}|\partial_{\alpha}\partial_{\beta}\ti{h}_{\ve}|^{2}.
\end{equation}
This shows that $\xi$ cannot have an interior minimum point, and so $\xi\geq0$ on $B_{R}\cap M$. Now one can use (\ref{appC2}) and a similar argument to Proposition \ref{boundary second order estimate} to show that $w\geq0$ on $B_{R}\cap M$, which implies the tangent-normal derivative estimate. We refer the reader to the proof of Proposition \ref{boundary second order estimate} for more details.

 Finally, the normal-normal derivative is bounded just by the fact that $\Delta_{\omega}h_{\ve}=-2n$. Thus, we have the second order estimate on the boundary:
\begin{equation*}
  |\nabla^{2}h_{\ve}|_{g} \leq Ce^{-C\psi} ~ \text{on $\de M$}.
\end{equation*}

To bound the Hessian everywhere now, we consider the quantity:
\begin{equation*}
  Q = e^{2B\psi}|\nabla^{2}h_{\ve}|_{g}^{2} + e^{B\psi}|\nabla h_{\ve}|_{g}^{2},
\end{equation*}
where $B$ is a constant to be determined. Let $x_{0}$ be an interior maximum point of $Q$. Choosing holomorphic normal coordinates for $g$ at $x_{0}$ gives
\begin{equation*}
  \sum_{i}(h_{\ve})_{k\ov{l}i\ov{i}} = (\Delta_{\omega}h_{\ve})_{k\ov{l}}-\de_{k}\de_{\ov{l}}(g^{i\ov{j}})(h_{\ve})_{i\ov{j}} = -\de_{k}\de_{\ov{l}}(g^{i\ov{j}})(h_{\ve})_{i\ov{j}},
\end{equation*}
which implies $|\sum_{i}(h_{\ve})_{k\ov{l}i\ov{i}}|\leq C|\nabla^{2}h_{\ve}|_{g}$. Similarly, we also have $|\sum_{i}(h_{\ve})_{kli\ov{i}}|\leq C|\nabla^{2}h_{\ve}|_{g}$. Combining this with the Cauchy-Schwarz inequality, it follows that
\begin{equation*}
\begin{split}
        & \Delta_{\omega}(e^{2B\psi}|\nabla^{2}h_{\ve}|_{g}^{2}) \\
\geq {} & 2e^{2B\psi}\textrm{Re}\left(\sum_{i,k,l}(h_{\ve})_{k\ov{l}i\ov{i}}(h_{\ve})_{\ov{k}l}+\sum_{i,k,l}(h_{\ve})_{kli\ov{i}}(h_{\ve})_{kl}\right) \\
        & +e^{2B\psi}|\nabla^{3}h_{\ve}|_{g}^{2} -4Be^{2B\psi}|\nabla\psi|_{g}|\nabla^{2}h_{\ve}|_{g}|\nabla^{3}h_{\ve}|_{g}
          -CB^{2}e^{(2B-C)\psi}|\nabla^{2}h_{\ve}|_{g}^{2} \\[2mm]
\geq {} & -Ce^{2B\psi}|\nabla^{2}h_{\ve}|_{g}^{2} -CB^{2}e^{2B\psi}|\nabla\psi|_{g}^{2}|\nabla^{2}h_{\ve}|_{g}^{2}
          -CB^{2}e^{(2B-C)\psi}|\nabla^{2}h_{\ve}|_{g}^{2} \\[2mm]
\geq {} & -CB^{2}e^{(2B-C)\psi}|\nabla^{2}h_{\ve}|_{g}^{2}.
\end{split}
\end{equation*}
Using the maximum principle and (\ref{part ii equ 4}), at $x_{0}$, we have
\begin{equation*}
\begin{split}
0 \geq {} & \Delta_{\omega}Q \\
    =  {} & \Delta_{\omega}(e^{2B\psi}|\nabla^{2}h_{\ve}|_{g}^{2})+\Delta_{\omega}(e^{B\psi}|\nabla h_{\ve}|_{g}^{2}) \\
  \geq {} &  -CB^{2}e^{(2B-C)\psi}|\nabla^{2}h_{\ve}|_{g}^{2}+\frac{1}{2}e^{B\psi}|\nabla^{2}h_{\ve}|_{g}^{2}-CB^{2}e^{(B-C)\psi}|\nabla h_{\ve}|_{g}^{2}.
\end{split}
\end{equation*}
After choosing $B\geq C$ sufficiently large, we see that
\begin{equation*}
CB^{2}e^{(2B-C)\psi} \leq \frac{1}{4}e^{B\psi} ~ \text{and} ~ e^{(B-C)\psi}|\nabla h_{\ve}|_{g}^{2}\leq C.
\end{equation*}
It then follows that
\begin{equation*}
  |\nabla^{2}h_{\ve}|_{g}^{2} \leq Ce^{-C\psi} \ \text{at $x_{0}$},
\end{equation*}
which implies $Q(x_{0})\leq C$, as desired.
\end{proof}

\begin{lemma}\label{applem 1}
Assume we are in the situation in Lemma \ref{applem}. For each integer $k\geq3$, there exist positive constants $B_{k}$, $C_{k}$ such that
\[
|\nabla^{k}h_{\ve}|_{g} \leq C_{k}e^{-B_{k}\psi}.
\]
\end{lemma}

\begin{proof}
Since $\psi$ is exponentially smooth, there exists a constant $C_{0}$ such that $e^{C_{0}\psi}$ is smooth. Using $\Delta_{\omega}h_{\ve}=-2n$, it is clear that
\begin{equation}\label{applem 1 equ 1}
\begin{split}
\Delta_{\omega}(e^{BC_{0}\psi}h_{\ve}) = {}
& -2ne^{BC_{0}\psi}+B(B-1)h_{\ve}e^{(B-2)C_{0}\psi}|\nabla(e^{C_{0}\psi})|_{g}^{2} \\
& +Bh_{\ve}e^{(B-1)C_{0}\psi}\Delta_{\omega}(e^{C_{0}\psi})
+2Be^{(B-1)C_{0}\psi}\textrm{Re}\langle \nabla(e^{C_{0}\psi}),\ov{\nabla}h_{\ve}\rangle,
\end{split}
\end{equation}
where $B$ is a constant to be determined. Using Lemma \ref{applem} and choosing $B$ sufficiently large, we obtain
\[
\|\Delta_{\omega}(e^{BC_{0}\psi}h_{\ve})\|_{C^{1}(M)} +
\|e^{BC_{0}\psi}\vp_{\ve}\|_{C^{3}(\de M)} \leq C.
\]
Applying the Schauder estimate, it follows that
\[
\|e^{BC_{0}\psi}h_{\ve}\|_{C^{2,\frac{1}{2}}(M)} \leq C.
\]
At the expense of increasing $B$, it follows from (\ref{applem 1 equ 1}) that
\[
\|\Delta_{\omega}(e^{BC_{0}\psi}h_{\ve})\|_{C^{1,\frac{1}{2}}(M)}+
\|e^{BC_{0}\psi}\vp_{\ve}\|_{C^{4}(\de M)} \leq C.
\]
Using the Schauder estimate again, we obtain
\[
\|e^{BC_{0}\psi}h_{\ve}\|_{C^{3,\frac{1}{2}}(M)} \leq C.
\]
Repeating the above argument, for any $k\geq 3$, there exist constants $B_{k}$, $C_{k}$ such that
\[
\|e^{B_{k}\psi}h_{\ve}\|_{C^{k,\frac{1}{2}}(M)} \leq C_{k},
\]
as required.
\end{proof}

\end{document}